\documentclass{amsart}

\usepackage{amssymb}
\usepackage{amsmath}

\usepackage{color}

\numberwithin{equation}{section}

\newtheorem{theorem}{Theorem}[section]
\newtheorem{lemma}[theorem]{Lemma}
\newtheorem{corollary}[theorem]{Corollary}
\newtheorem{prop}[theorem]{Proposition}
\newtheorem{remark}[theorem]{Remark}
\newtheorem{assumption}[theorem]{Assumption}

\newcommand\pd{{\partial}}

\newcommand\tv{{\tilde{v}}}
\newcommand\w{{\tilde{w}}}

\newcommand\W{{\widetilde{W}}}
\newcommand{\R}{{\mathbb{R}}}
\newcommand{\N}{{\mathbb{N}}}

\newcommand\cB{{\mathcal B}}

\newcommand\cA{{\mathcal A}}
\newcommand\cD{{\mathcal D}}
\newcommand\cT{{\mathcal T}}
\newcommand\lam{{{\lambda, l}}}

\newcommand\wt{{\langle z\rangle}}

\begin{document}

\title[Prandtl equations in three space variables]
{A well-posedness Theory for the Prandtl equations in three space variables}

\author{C.-J. Liu}
\address{Cheng-Jie Liu
\newline\indent
Department of Mathematics, Shanghai Jiao Tong University
\newline\indent
Shanghai, 200240, P. R. China}
\email{cjliusjtu@163.com}

\author{Y.-G. Wang}
\address{Ya-Guang Wang
\newline\indent
Department of Mathematics, and MOE-LSC, Shanghai Jiao Tong University
\newline\indent
Shanghai, 200240, P. R. China}
\email{ygwang@sjtu.edu.cn}

\author{T. Yang}
\address{Tong Yang
\newline\indent
Department of mathematics, City University of Hong Kong,
\newline\indent
Hong Kong, P. R. China, and
\newline\indent
Department of Mathematics, Shanghai Jiao Tong University
\newline\indent
Shanghai, 200240, P. R. China}
\email{matyang@cityu.edu.hk}


\date{}



\begin{abstract}
The well-posedness of the three space dimensional
 Prandtl equations is studied under some constraint on its flow structure. It reveals
that the classical Burgers equation plays an important role in determining this
type of  flow with special structure, that avoids the appearance of the complicated secondary flow in the three-dimensional Prandtl boundary layers.
And the sufficiency of the monotonicity condition on the tangential velocity field for the existence
of solutions to the Prandtl boundary layer equations is illustrated in the three dimensional setting. Moreover, it is shown that this structured flow is linearly stable for any three-dimensional perturbation.
\end{abstract}

\maketitle

\tableofcontents


\section{Introduction}

To describe the behavior of viscous flows in a neighborhood of physical boundary qualitatively and quantitatively is a classical problem both in theoretical and applied fluid mechanics. It was observed by L. Prandtl in his seminal work \cite{prandtl} that, away from the boundary the flow is mainly driven by convection
so that the viscosity can be negligible, while
in a small neighborhood
of physical boundary the effect of the viscosity plays a significant role in the flow. Hence, there exists a thin transition layer near the boundary, in which the behavior of flow changes dramatically, this transition layer is so-called the boundary layer.

Mathematically, taking the incompressible Navier-Stokes equations
as the governed system for the viscous flow with velocity being non-slip on the boundary,
in Prandtl's theory, letting $\epsilon$ be the viscosity coefficient,
outside the layer of thickness $\sqrt{\epsilon}$ near the boundary, the flow is approximated by an inviscid one, and it is basically governed by the incompressible Euler equations; on the other hand, inside the layer, the convection and the viscosity balance so that the flow can be modelled by a system derived from the Navier-Stokes equations by asymptotic expansion, that is, the Prandtl boundary layer equations. The formal derivation of the Prandtl equations can be found in \cite{prandtl}, for example. 

In the Prandtl boundary layer equations, the tangential velocity profile
 satisfies a system
of nonlinear degenerate parabolic equations, and the incompressibility of flow still holds in the layer, so
 the tangential and normal velocities are coupled by the divergence-free constraint.  The main difficulties in studying the Prandtl equations lie in the degeneracy, mixed type, nonlinearity and non-local effect in the system, so that the classical mathematical theories of partial differential equations can hardly be applied. For this, in more than one hundred years since the Prandtl equations
were derived, there is still no general mathematical theory on the well-posedness nor a rigorous justification of the viscous limit of the Navier-Stokes equations to the superposition of the Prandtl and Euler equations except in the framework
of analytic functions by using the abstract Cauchy-Kowaleskaya theory (cf. \cite{CS, cannone, Samm} etc.)
or under the assumption that the vorticity of the Euler flow is supported away from the boundary (\cite{me}).
 However, the analytic property rules
out the physical singularity, so more physical function spaces for solutions need to be sought.

 On the other hand, under the monotonicity condition on the tangential velocity, local well-posedness was obtained in two space dimension in the classical work by Oleinik and her collaborators (\cite{OA1, Ole}), and then the global existence of a weak solution with extra
favorable condition on pressure by Xin and Zhang in \cite{XZ}.
 These existence results rely on the Crocco transformation which transfers
the degenerate and mix-typed system to a
scalar degenerate parabolic equation in
two dimensional case. Motivated by the fact that energy method can be well
applied to the Navier-Stokes equations,  a
new approach was introduced in \cite{AWXY}
to study the well-posedness theory  in
Sobolev spaces by using  a direct energy method without using the Crocco transformation. A similar result was
also obtained in \cite{Masmoudi-Wong}.

We would like to emphasize that
there is basically no  well-posedness theory for the three dimensional Prandtl equations
except the analytic case \cite{Samm},
mainly due to the extra difficulties coming from secondary flow appeared in the three dimensional boundary layers (\cite{moore}) and the complicated structure of boundary layers arising from the multi-dimensional velocity fields. Indeed, the well-posedness of the Prandtl equations in three space variables is one
of the important open questions proposed by Oleinik and Samokhin on page 500 in their classical monograph \cite{Ole}.

The main purpose of this paper is to study the well-posedness
in the function spaces of finite smoothness,
of the initial-boundary value problem for the three dimensional Prandtl equations in the domain $\{t>0, (x,y)\in D, z>0\}$ for a fixed $D\subset \R^2$, that is,
\begin{equation}\label{1.1}
\begin{cases}
\partial_t u+(u\partial_x+v\partial_y+w\partial_z) u+\partial_x p=\partial_z^2u,\\
\partial_t v+(u\partial_x+v\partial_y+w\partial_z) v+\partial_y p=\partial_z^2v,\\
\partial_x u+\partial_y v+\partial_z w=0,\\
(u,v,w)|_{z=0}=0, \quad \lim\limits_{z\to+\infty}(u,v)=(U(t,x,y), V(t,x,y)),
\end{cases}
\end{equation}
where $(U(t,x,y), V(t,x,y))$ and $p(t,x,y)$ are the tangential velocity fields and pressure on the boundary $\{z=0\}$ of the Euler flow, satisfying
\begin{equation}\label{euler}
\begin{cases}
  \partial_t U+U\partial_x U+V\partial_y U+\partial_x p=0,\\
  \partial_t V+U\partial_xV+V\partial_yV+\partial_y p=0.
  \end{cases}
\end{equation}

One of the key observations in this paper is that a special structure of the Euler
flow and the initial-boundary conditions can lead to the existence of a solution to the three dimensional
Prandtl equations with the same structure. Even though the existence of
this kind of three dimensional boundary layer relies on the structure condition,
it does give an existence theory for the three dimensional Prandtl system for
which almost no other mathematical theory is known so far. Moreover, the monotonicity of
the tangential velocity in the normal direction that is better understood in two dimensional space
can now be illustrated in the three dimensional problem. In addition, it is
interesting to find out that the classical Burgers equation plays an important
role in constructing this kind of flow with structure.

Precisely, without loss of generality, assume that the outer Euler flow takes the following form on the boundary $\{z=0\}$,
\begin{equation}\label{1.3}
\big(U(t,x,y), k(t,x,y)U(t,x,y), 0; ~p(t,x,y)\big),
\end{equation}
with $U(t,x,y)>0$.
We are trying to construct  a solution of the three dimensional Prandtl equations \eqref{1.1}  with the same structure
\begin{equation}\label{1.4}
\big(u(t, x,y,z), k(t,x,y)u(t,x,y,z), w(t,x,y,z)\big),\end{equation}
with $u(t,x,y,z)$ being strictly increasing in $z>0$. If this kind flow exists, then the special form \eqref{1.4} of the boundary layer profile shows that the direction of the tangential velocity field in the boundary layer is invariant in the normal variable $z$, consequently the secondary flow does not appear. Plugging the form \eqref{1.4} into the second equation in \eqref{1.1}, we get
$$
\partial_t (ku)+ (u\partial_x +ku\partial_y +w
\partial_z)(ku)+\partial_y p-k\partial_z^2 u=0,$$
which implies
\begin{equation}\label{eqk}
 u\left[\partial_t
k+u(\partial_x+k\partial_y)k\right]-k\pd_xp+\pd_yp=0,
\end{equation}
by using the first equation of \eqref{1.1}.

Noting that $k(t,x,y)$ is independent of $z$, by differentiating \eqref{eqk} with respect to $z$, it follows
\[
\partial_z u\partial_tk+2u\partial_zu
(\partial_x+k\partial_y)k=0,
\]
which implies
\begin{equation}\label{1.6}
  \pd_t k+2u(\partial_x +k\partial_y)k=0,
\end{equation}
where we have used the fact  that $\pd_z u>0$. Differentiating \eqref{1.6} with respect to $z$ gives the Burgers equation
\begin{equation}\label{eqkk}
 (\pd_x+k\pd_y)k=0.
\end{equation}

Combining \eqref{1.6} with \eqref{eqkk}, we get $\partial_tk=0$.
Plugging these equalities into \eqref{eqk}, it follows
\begin{equation}\label{eq_p}
\pd_yp-k\pd_xp=0, \end{equation}
which means that $(\pd_x p,\pd_y p)$ is parallel to both the velocity field of out Euler flow and the tangential velocity field in the boundary layer.

Based on the above calculation, from now on, we impose the following condition (H) on the outer flow and the function $k$:

\begin{enumerate}
\item[(H1)] in the domain $\{t>0, (x,y)\in D, z>0\}$ with
    a smooth bounded region $D\subset \R^2$, the function $k$ depends on $(x,y)$ only, and satisfies the Burgers equation \eqref{eqkk} in $D$. 
  Furthermore, the set $\overline{\gamma_-}\setminus\gamma_-$ contains finite number of points, where
    \begin{equation}\label{def_gam}
    \gamma_-~=~\{(x,y)\in\pd D\big|~(1,k(x,y))\cdot\vec{n}(x,y)<0\},
    \end{equation}
 with $\vec{n}(x,y)$ being the unit outward normal vector of $D$ at $(x,y)\in\pd D$, and $\overline{\gamma_-}$ is the closure of $\gamma_-$ on the boundary $\pd D$;

\item[(H2)] the Euler flow
$$\big(U(t,x,y), k(x,y)U(t,x,y), 0, p(t,x,y)\big)$$
with $U(t,x,y)>0$, satisfies
\begin{equation}\label{eqU}
\begin{cases}
\pd_tU+U\pd_xU+kU\pd_yU+\pd_x p=0,\\
\pd_y p-k\pd_x p=0.
\end{cases}
\end{equation}
\end{enumerate}

The main problem (MP) to be studied in this paper can be formulated as follows.

\begin{enumerate}
\item[(MP)] Under the above assumption (H), to study the well-posedness for the following problem of the Prandtl equations
in the domain $Q_T=\{0<t\leq T, (x,y)\in D,  z>0\}$:
\begin{equation}\label{pdtl-1}
\begin{cases}
\partial_t u+(u\partial_x +v\partial_y +w
\partial_z) u-\partial_z^2 u=-\pd_x p,\\[2mm]
\partial_t v+ (u\partial_x +v\partial_y +w
\partial_z)v-\partial_z^2 v=-\pd_y p,\\[2mm]
\partial_x u+\partial_y v+\partial_z w=0,\\[2mm]
u|_{z=0}=w|_{z=0}=0, \quad \lim\limits_{z\rightarrow +\infty}(u,v)=(U(t,x,y), k(x,y)U(t,x,y)),\\[2mm]
(u,v)|_{\pd Q_T^-}=(u_1(t,x,y,z),k(x,y)u_1(t,x,y,z)),\\[2mm]
(u, v)|_{t=0}=(u_0(x,y,z), k(x,y)u_0(x,y,z)),
\end{cases}
\end{equation}
where $\pd Q_T^-=(0,T]\times\gamma_-\times\mathbb{R}_+$
with $\gamma_-$ being given in \eqref{def_gam}.
\end{enumerate}

\vspace{.1in}
The main results on the well-posedness of the initial boundary value problem \eqref{pdtl-1} in given in the following theorem.

\begin{theorem}[Main Result] Under the above conditions $(H1)$-$(H2)$ with $k\in C^{10}(D)$ and $(U,p)\in C^{10}\big((0,T]\times D\big)$ for a fixed $T>0$, assume that
$$u_0\in C^{15}(D\times\R_z^+),\quad~u_1\in C^{15}(\pd Q_T^-),$$
have the following properties:
\begin{enumerate}
\item[(1)]  $\pd_z u_0>0,\pd_z u_1>0$ for all $z\geq0$, and there is constant $C_0>0$ such that
$$C_0^{-1}\Big(U(0,x,y)-u_1(x,y,z)\Big)~\leq~\pd_z u_0(x,y,z)~\leq~ C_0\Big(U(0,x,y)-u_0(x,y,z)\Big),
$$
and
$$\begin{array}{ll}
C_0^{-1}\Big(U(t,x,y)-u_1(t,x,y,z)\Big)~& \leq~\pd_z u_1(t,x,y,z)\\[2mm]
&\leq~ C_0\Big(U(t,x,y)-u_1(t,x,y,z)\Big)
\quad {\rm on}\quad \pd Q_T^-;
\end{array}
$$
\item[(2)] the compatibility conditions hold up to order $6$ at $\{t=0\}\cap \pd Q_T^-$, and the compatibility conditions hold up to order $3$ ($4$ resp.) at $\{t=0\}\cap \{z=0\}$ ($\{t=0\}\cap \{z=\infty\}$ resp.), and $\pd Q_T^-\cap\{z=0\}$ ($\pd Q_T^-\cap\{z=\infty\}$ resp.).
\end{enumerate}
Then,
there exist $0<T_0\le T$ and a unique classical solution $(u,v,w)$ to the problem \eqref{pdtl-1} in the
domain $Q_{T_0}$, moreover, the solution is linearly stable with respect to any three-dimensional smooth perturbation of the initial data and boundary data
without the special structure given in \eqref{pdtl-1}.
\end{theorem}

\begin{remark}
One important observation on the problem \eqref{pdtl-1} is that for classical solutions, under the
assumption (H), the solution to the problem \eqref{pdtl-1} satisfies
$v(t,x,y,z)=k(x,y)u(t,x,y,z)$, i.e. the boundary layer flow has the special structure as given in \eqref{1.4}. Indeed,  assuming that $(u,v,w)$ is a  classical solution to \eqref{pdtl-1}, then $W(t,x,y,z)=v(t,x,y,z)-k(x,y)u(t,x,y,z)$ satisfies the following problem:
\begin{equation}\label{pdtl-2}\begin{cases}
\partial_t W+(u\partial_x +ku\partial_y +w
\partial_z) W-\partial_z^2 W+(\pd_y v-k\pd_y u) W=0,\\
W|_{z=0}=0,\quad \lim\limits_{z\rightarrow+\infty}W=0,\\
W|_{\pd Q_T^-}=0,\quad W|_{t=0}=0,
\end{cases}\end{equation}
which has only trivial solution $W\equiv 0$ by using the energy argument. Therefore, to study the problem \eqref{pdtl-1} is equivalent to study the following reduced problem for only
two unknown functions $u$ and $w$ in $Q_T$,
\begin{equation}\label{eqk2}\begin{cases}
\partial_t u+(u\partial_x +ku\partial_y +w
\partial_z) u-\partial_z^2 u=-\pd_x p,\\[2mm]
\partial_x u+\partial_y (ku)+\partial_z w=0,\\[2mm]
u|_{z=0}=w|_{z=0}=0, \quad \lim\limits_{z\rightarrow +\infty}u=U(t,x,y),\\[2mm]
u|_{\pd Q_T^-}=u_1(t,x,y,z),\quad u|_{t=0}=u_0(x,y,z).
\end{cases}\end{equation}
\end{remark}

In the rest of  this paper, we will focus on
 the well-posedness of the problem \eqref{eqk2} under the assumption (H) and the monotonic condition $\pd_zu_0>0,\pd_zu_1>0$
 for all $z>0$. Precisely,
in Section 2, motivated by the work of Oleinik and her collaborators,
we prove the local existence of a classical solution to the problem \eqref{eqk2} under certain smoothness and compatibility conditions of the initial and boundary data by using the Crocco transformation.
Moreover, in Section 3, by adopting the approach
given in \cite{AWXY},, we deduce that the structured classical solution constructed in Section 2 is linearly stable with respect to any three dimensional perturbation for the Prandtl boundary layer equations.
Finally, in Section 4, we present the main arguments of the construction of approximate solutions to the problem derived from the reduced problem  \eqref{eqk2} after taking the Crocco transformation.

Note that under the additional favorable assumption on the pressure of the outer flow, that is $\partial_xp(t, x,y)\le 0$ for $t>0$ and $(x,y)\in D$,  as in  \cite{XZ} for the two dimensional Prandtl equations, global existence of  weak solution for the problem  \eqref{eqk2} by using Crocco transformation can be obtained and this will be presented
in our coming paper. In this paper, we will focus on the existence of classical solution together with
its stability.

Before the end of the introduction, in addition to the well-posedness results mentioned above, let us review some other works on the
Prandtl equations. Without the monotonicity assumption, it is well expected that
singularities will develop in the Prandtl equations.
Van Dommelen and Shen in \cite{van} illustrated the ``Van  Dommelen singularity'' by
considering an implusively started circular cylinder to show
the blowup of the normal velocity, and
 E and Enquist in \cite{e-2} precisely constructed some finite time blowup
solutions to the two-dimensional Prandtl equations.
There are also some interesting works on the
instability of the two-dimensional Prandtl equations, in particular
 in the Sobolev spaces. Corresponding to the well known
Rayleigh criterion for the Euler flow,  Grenier \cite{grenier} showed that
 the unstable Euler shear flow  yields instability of the Prandtl equations.
It was shown in \cite{GV-D}
 that a non-degenerate
critical point in  the shear flow of the Prandtl equations leads to
  a strong linear ill-posedness of the Prandtl equations in the
Sobolev space framework. Moreover, \cite{GV-N} strengthens
the result of \cite{GV-D}  for an unstable
shear flow.
Furthermore, the ill-posedness in the nonlinear
setting was proved in \cite{guo} to show that
the  Prandtl equations are ill-posed near non-stationary and non-monotonic shear flows so that
the asymptotic boundary-layer
expansion is not valid for non-monotonic shear layer flows in Sobolev spaces.


\section{Local existence of classical solutions}


\subsection{Crocco transformation, assumptions and iteration scheme}

For a fixed bounded domain $D$ of $\R^2$ with a smooth boundary $\pd D$, denote by $Q$ the domain $\{(t,x,y,z)|~0< t<T,(x,y)\in D,z\in\R_+\}$.
Consider the following problem derived from the three dimensional Prandtl problem in the domain $Q_T$,
\begin{equation}\label{eqk3}
\begin{cases}
\partial_t u+(u\partial_x +ku\partial_y +w
\partial_z) u-\partial_z^2 u=-\pd_x p,\\[2mm]
\partial_x u+\partial_y (ku)+\partial_z w=0,\\[2mm]
u|_{z=0}=w|_{z=0}=0, \quad \lim\limits_{z\rightarrow +\infty}u=U(t,x,y),\\[2mm]
u|_{\pd Q_T^-}=u_1(t,x,y,z),\quad
u|_{t=0}=u_0(x,y,z),
\end{cases}
\end{equation}
with the same notations as given in \eqref{eqk2}.

Assuming that $U(t,x,y)>0$ for all $t>0$ and $(x,y)\in D$,
we are going to construct a solution to the problem \eqref{eqk3} with the $x-$direction tangential velocity $u(t,x,y,z)$ being strictly monotone in $z>0$,
under the assumption:
\begin{equation}\label{ass_ib}
\pd_z u_0>0,~\pd_z u_1>0,~\quad {\rm for}~z\geq0.
\end{equation}

\vspace{0.5cm}

\underline{Crocco Transformation:} Inspired by the method introduced in \cite{OA1},
apply the following Crocco transformation to the problem \eqref{eqk3},
\begin{equation}\label{crocco}
\xi=x,~\eta=y,~\zeta=\frac{u(t,x,y,z)}{ U(t,x,y)},
\end{equation}
and let $W(t,\xi,\eta,\zeta)=\frac{\pd_z u(t,x,y,z)}{U(t,x,y)}.$
Obviously, when the unknown function $u$ is strictly increasing in $z$, the transformation \eqref{crocco} is invertible, and under  this transformation, the original domain $Q_T=\{(t,x,y,z)|~0< t\leq T,(x,y)\in D,z\in\R_+\}$ is transformed into
$$\Omega=\{(t,\xi,\eta,\zeta)|~0< t\leq T,~(\xi,\eta)\in D,~0<\zeta<1\}.
$$

Therefore,  to solve the problem \eqref{eqk3} is reduced to find a solution $W(t,\xi,\eta,\zeta)$ to the following initial boundary value problem in $\Omega$,
\begin{equation}\label{eq_tr}
\begin{cases}
L(W)~\triangleq~\partial_t W+\zeta U(\partial_\xi +k\partial_\eta)W+A\pd_\zeta W+BW-W^2\pd_\zeta^2 W=0,\\
W|_{\zeta=1}=0,\quad W\pd_\zeta W|_{\zeta=0}=\frac{ p_x}{U},\\
W|_{\Gamma_-}=W_1(t,\xi,\eta,\zeta)\triangleq \frac{\pd_z u_1}{U},\\
W|_{t=0}=W_0(\xi,\eta,\zeta)\triangleq \frac{\pd_z u_0}{U},
\end{cases}
\end{equation}
where
$$A=-\zeta(1-\zeta){ U_t\over U}-(1-\zeta^2){ p_x\over U},\quad
~B={ U_t\over U}+\zeta(U_x+kU_y)-\pd_y k\cdot \zeta U,$$
and
$$\Gamma_-=\{(t,\xi,\eta,\zeta):~0< t\leq T,~(\xi,\eta)\in\gamma_-,~0<\zeta<1\}.$$
\\

\underline{Notations and Assumptions:}
First, we introduce some notations defined on $\pd D$: denote by $\vec{\tau}(\xi,\eta)$ and $\vec{n}(\xi,\eta)$ the unit tangential and outward normal vectors on $\pd D$ at $(\xi,\eta)\in\pd D$, and
\[\pd_\tau=\vec{\tau}\cdot(\pd_\xi,\pd_\eta),~\pd_n=\vec{n}\cdot(\pd_\xi,\pd_\eta).\]

Obviously, the operator $\pd_\xi+k\pd_\eta$ restricted on $\pd D$ can be rewritten as:
\begin{equation}\label{def_k}
\pd_\xi+k\pd_\eta~=~k_\tau\pd_\tau+k_n\pd_n,
\end{equation}
with
\[k_\tau(\xi,\eta)=\left(1,k(\xi,\eta)\right)\cdot\vec{\tau}(\xi,\eta),
\quad~k_n(\xi,\eta)=\left(1,k(\xi,\eta)\right)\cdot\vec{n}(\xi,\eta).\]

To state the compatibility conditions of the initial and boundary data of the problem \eqref{eqk3}, denote
by
\begin{equation}\label{w-0-1}
W_0^i(\xi,\eta,\zeta)~=~\pd_t^i W|_{t=0},\quad~W_1^i(t,\xi,\eta,\zeta)~=~\pd_n^i W|_{\Gamma_-}\end{equation}
for $0\leq i\leq4$.

Obviously, by using the equation given in \eqref{eq_tr}, we can easily represent $W_0^{i+1}$ and $W_1^{i+1}$ by using $W_0$ and $W_1$ inductively for $0\leq i\leq3$, that is,
\begin{equation}\label{def_w0}
\begin{split}
W_0^{i+1}(\xi,\eta,\zeta)~=~&-\sum_{j=0}^{i}C_{i}^j\cdot\Big\{\zeta
\pd_t^{i-j}U\cdot(\pd_\xi+k\pd_\eta)W_0^j+\pd_t^{i-j}A\cdot\pd_\zeta W_0^j\\
&\quad+\pd_t^{i-j}B\cdot W_0^j-\pd_\zeta^2W_0^{i-j}\cdot\Big[\sum_{l=0}^jC_j^l~W_0^l\cdot W_0^{j-l}\Big]\Big\}
\end{split}
\end{equation}
at $t=0$,
 and
\begin{equation}\label{def_w1}
(\zeta U k_n)\cdot W_1^{i+1}=f_{i+1}
\end{equation}
on the boundary $\Gamma_-$, where the function $f_{i+1}$, defined on the boundary $\Gamma_-$, is given by
\begin{equation}\label{def_w1-1}
\begin{split}
f_{i+1}=& -\pd_t W_1^{i}-\sum_{j=0}^{i-1}C_{i}^{j}\Big[\zeta \pd_n^{i-j}(Uk_n)\cdot W_1^{j+1}\Big]-\sum_{j=0}^{i}C_{i}^j\Big\{\zeta\pd_n^{i-j}(Uk_\tau)\cdot\pd_\tau W_1^j\\
&\quad+\pd_n^{i-j}A
\cdot\pd_\zeta W_1^j+\pd_n^{i-j}B\cdot W_1^j-\pd_\zeta^2 W_1^{i-j}\cdot\Big[\sum_{l=0}^jC_j^lW_1^l\cdot W_1^{j-l}
\Big]\Big\},
\end{split}\end{equation}
with $C_i^j={i!\over j!(i-j)!}$ for integer $0\leq j\leq i$.

Now, we give the following assumptions on the initial and boundary data of \eqref{eq_tr}.

\begin{assumption}\label{ass_1}
Assume that for the problem \eqref{eq_tr},
\begin{equation}\label{reg_k}
k\in C^{10}(D),\quad (U,p)\in C^{10}\big((0,T]\times D\big),\
\end{equation}
and the initial boundary data
\begin{equation}\label{reg_ini}
W_0\in C^{14}\big(D\times(0,1)\big),\quad~W_1\in C^{14}(\Gamma_-),
\end{equation}
such that we have the following properties:

(1) there is a constant $M>0$, such that
\begin{equation}\label{seq_ass}
M^{-1}(1-\zeta)\leq W_0(\xi,\eta,\zeta),W_1(t,\xi,\eta,\zeta)\leq M(1-\zeta),
 \end{equation}

(2) functions $W_1^i\in C^6(\Gamma_-)~(1\leq i\leq 4)$,  and the following compatibility conditions hold:
\[\begin{split}
\qquad\qquad{\bf{(i)}}&~\pd_\zeta^m\pd_\tau^l\pd_n^jW_0^i|_{\Gamma_-}~=~\pd_\zeta^m\pd_\tau^l\pd_t^iW_1^j|_{t=0},\quad ~~{\rm for}~i,j\leq4,~m+l+j+i\leq5,\\
\qquad\qquad{\bf{(ii)}}&~\pd_\xi^j\pd_\eta^l W_0^i|_{\zeta=1}=\pd_t^q\pd_\tau^r W_1^s|_{\zeta=1}=0,\quad  {\rm for}~i+j+l\le 3, q+r+s\leq3,\\
\qquad\qquad{\bf{(iii)}}&~\pd_{\xi,\eta}^\alpha\pd_\zeta\left(\sum_{j=0}^iC_i^jW_0^j\cdot W_0^{i-j}\right)|_{\zeta=0}~=~2\pd_{\xi,\eta}^\alpha\pd_t^i({p_x\over U})|_{t=0},\quad {\rm for} ~|\alpha|+i\leq2,\\
&~\pd_t^i\pd_\tau^j\pd_\zeta\left(\sum_{l=0}^m C_m^lW_1^l\cdot W_1^{m-l}\right)|_{\zeta=0}~=~2\pd_t^i\pd_\tau^j\pd_n^m({p_x\over U})|_{\Gamma_-},~ {\rm for} ~i+j+m\leq2,\\
&{\rm where}~\alpha=(\alpha_1,\alpha_2),~|\alpha|=\alpha_1+\alpha_2.
\end{split}\]
\end{assumption}

\begin{remark}
(1) The above regularity assumption on $k,U,p,W_0$ and $W_1$ given in \eqref{reg_k} and \eqref{reg_ini} respectively, and the compatibility condition 2(i) are for the requirement that the zero-th order approximate solution $W^*$ constructed in \eqref{w*} and \eqref{w*1} needs to be $C^6$ in a neighborhood of the boundary, which implies $F\in W^{4,\infty}$ for the function $F$ given in \eqref{pr_n}, to have the boundedness of approximate solutions $W^n_\epsilon$ determined by \eqref{pr_n} in $W^{4,\infty}$ uniformly in $\epsilon$. The compatibility condition 2(ii)-(iii) is to guarantee the approximate solutions constructed by \eqref{eq_seq} satisfying $W^n\in C^3(\overline{\Omega})$ for all $n\ge 0$. 

(2)
From \eqref{def_w1} and \eqref{def_w1-1}, we know that by
the assumptions given  above on the boundary data $W_1$,
the functions $f_i,~1\leq i\leq4$ defined on $\Gamma_-$ satisfy
\begin{equation}
f_i~=~\mathcal{O}(\zeta k_n),\quad as~\zeta k_n~\rightarrow~0.
\end{equation}

(3) It is easy to see that Assumption \ref{ass_1} can be derived from the corresponding conditions for the
original problem \eqref{eqk3} of the Prandtl equations, which will be given in \S2.5 later.
\end{remark}

\vspace{0.5cm}

\underline{Iteration scheme for solving the problem \eqref{eq_tr}.}

Let $W^0(t, \xi, \eta, \zeta)$ be the zero-th order approximate solution of the problem \eqref{eq_tr}, which will be constructed in 
Section 4,
such that $W^0$ has bounded derivatives up to
 order four in $\overline{\Omega}$, and satisfies
\begin{equation}\label{w0}
\begin{cases}
\pd_t^i W^0|_{t=0}=W_0^i(\xi,\eta,\zeta),~\pd_n^j W^0|_{\Gamma_-}=W_1^j(t,\xi,\eta,\zeta),\quad for~0\leq i,j\leq  3,\\[2mm]
M^{-1}(1-\zeta)\leq W^0(t,\xi,\eta,\zeta)\leq M(1-\zeta),\quad \forall (t,\xi,\eta,\zeta)\in\Omega,
\end{cases}
\end{equation}
for the positive constant $M$ given in \eqref{seq_ass}.

Then, we construct the $n-$th order approximate solution of  \eqref{eq_tr}  by solving the following linearized
 problem in $\Omega$,
\begin{equation}\label{eq_seq}
\begin{cases}
L_n(W^n)~\triangleq~\partial_t W^n+\zeta U(\partial_\xi +k\partial_\eta)W^n+A\pd_\zeta W^n+BW^n-(W^{n-1})^2\pd_\zeta^2 W^n=0,\\
 W^{n-1}\pd_\zeta W^n|_{\zeta=0}={ p_x\over U},
~W^n|_{\Gamma_-}= W_1(t,\xi,\eta,\zeta),\\
W^n|_{t=0}= W_0(\xi,\eta,\zeta).
\end{cases}\end{equation}
Note that  we do not need to impose any condition of $W^n$ on the boundary $\{\zeta=1\}$,
as we shall verify in Proposition \ref{lem_n3} that the approxiamte solution $W^n$ vanishes on $\{\zeta=1\}$  for all $n\geq1$ by induction on $n$.

In the following subsection, assuming that the approximate solution sequence $\{W^n\}_{n\geq0}$ has been constructed and $W^n$ has continuous and
bounded derivatives up to order three in $\overline{\Omega}$, let us show that when $n\rightarrow+\infty$, $W^n$ converges to a
classical solution of the problem \eqref{eq_tr} in $\Omega$ with $0\leq t\leq t_1$ for some
$0<t_1\leq T$.
The construction of the approximate solution $W^n$ to the problem \eqref{eq_tr} will be given in Section 4.


\subsection{Classical solution of the problem transformed by Crocco transformation}

The purpose of
this subsection is  to
prove the convergence of the
iteration
 scheme \eqref{eq_seq}, which
gives
the existence of
a classical
solution to the nonlinear problem \eqref{eq_tr}.

Firstly, we have the following
comparison principle for the problem \eqref{eq_seq}.

\begin{lemma}\label{lem_seq}
Assume that 
$W^n \in C^2(\Omega)$
is the solution of \eqref{eq_seq} obtained in Proposition \ref{lem_n3},
with
$W^n|_{\zeta=1}=0$,
 and $W^{n-1}|_{\zeta=0}>0$.

(1) If a smooth function $V$ satisfies $L_n(V)\leq0$ in $\Omega$,
with $L_n(\cdot)$ being the operator given in the problem \eqref{eq_seq},
 and
\[V|_{t=0}\leq W_0,~V|_{\Gamma_-}\leq W_1,~V|_{\zeta=1}\leq0,~W^{n-1}\pd_\zeta V|_{\zeta=0}
\ge
{p_x\over U},\]
on the boundary of $\Omega$, then we have
$V\leq W^n$ in $\Omega.$

(2) If $V$ satisfies $L_n(V)\geq0$ in $\Omega$ and
\[V|_{t=0}\geq W_0,~V|_{\Gamma_-}\geq W_1,~V|_{\zeta=1}\geq0,~W^{n-1}\pd_\zeta V|_{\zeta=0}\leq {p_x\over U},\]
then $V\geq W^n$ holds in $\Omega.$
\end{lemma}

\begin{proof}[\bf{Proof.}]
A similar comparison principle was given in  \cite[Lemma 4.3.1]{Ole} for the two-dimensional problem, here the main difference is there is an additional boundary $\Gamma_-$ in the problem \eqref{eq_seq}, so for completeness,
we will only give the main steps of the proof for the first case, and one can study the second case similarly.

Set $u~\triangleq~W^n-V$. From the assumption we have
\begin{equation*}
\begin{split}
~L_n(u)=L_n(W^n)-L_n(V)\geq0,\quad in ~\Omega,\\
~u\geq0,\quad on ~\{t=0\}\cup\Gamma_-\cup\{\zeta=1\},\\
~W^{n-1}\pd_\zeta u\le 0,\quad on ~\{\zeta=0\}.
\end{split}
\end{equation*}
Let $w=ue^{-b t}$ with a constant $b$ satisfying
$|B(t,\xi,\eta,\zeta)|\leq b$ in $\Omega$.  Then,
\begin{equation*}
\begin{split}
~L_n(w)+bw=L_n(u)e^{-b t}\geq0,\quad in ~\Omega,\\
~w\geq0,\quad on ~\{t=0\}\cup\Gamma_-\cup\{\zeta=1\},\\
~W^{n-1}\pd_\zeta w\le 0,\quad on ~\{\zeta=0\}.
\end{split}
\end{equation*}
Applying the maximum principle
of degenerate parabolic operators
to the above problem, it follows that $w$ does not
 attain its negative minimum
in the interior of $\Omega$,
on the plan $\{t=T\}$,
and at $\{t=0\}\cup\{\zeta=1\}$.

From the boundary condition and $W^{n-1}|_{\zeta=0}>0$, we have $\pd_\zeta w|_{\zeta=0}\le 0$, which implies
that $w$ does not have any negative minimal point on the boundary $\{\zeta=0\}$.

On the other hand, if $w$ attains its negative minimum at a point $P$ on the boundary $\{(\xi,\eta)\in\pd D\}$, then at this point, $\pd_t w=\pd_\zeta w=0,~\pd_\zeta^2 w\geq0$ and
\[\pd_\tau w=\vec{\tau}\cdot\nabla_{(\xi,\eta)}w=0,\quad \pd_n w=\vec{n}\cdot\nabla_{(\xi,\eta)}w\leq0,\]
with $\vec{\tau}(\xi,\eta)$ and $\vec{n}(\xi,\eta)$ being the unit tangential and outward normal vectors at $(\xi,\eta)\in\pd D.$ By using \eqref{def_k}, we have
\[(\pd_\xi+k\pd_\eta)w=k_\tau\vec{\tau}\cdot\nabla_{(\xi,\eta)}w+k_n\vec{n}\cdot\nabla_{(\xi,\eta)}w.\]
From the equation of $w$, we know that at the negative minimum point $P$, $(\pd_\xi+k\pd_\eta)w\ge 0$, which implies that $P\in\Gamma_-$.
This is a contradiction to $w|_{\Gamma_-}\geq0$.

Hence,
in the whole $\Omega$,
$w\geq0$, which implies $W^n\geq V$ in $\Omega$.
\end{proof}

To show that $W^n$ is uniformly bounded  in $n$, we first define two smooth functions:
\begin{equation}\label{def_V1-2}
V_1(t,\zeta)~=~m\varphi(\zeta) e^{-\alpha t},\quad V_2(t,\zeta)~=~C(1-\zeta)e^{\beta t},
\end{equation}
where,
\begin{equation}\label{def_phi}
\varphi(\zeta)=\begin{cases}\quad e^{\alpha_1\zeta},\quad&0\leq\zeta<\delta_0,\\
{\rm smooth~connection},\quad &\delta_0\leq\zeta\leq1-\delta_0,\\
\quad 1-\zeta,\quad &1-\delta_0<\zeta\leq1,
\end{cases}
\end{equation}
has bounded first and second derivatives, and
\[\delta_0\leq\varphi(\zeta)\leq2,\qquad {\rm for~all}~0\leq\zeta\leq1-\delta_0.\]
Here,
the positive constants $\delta_0,\alpha_1,m,C,\alpha,\beta$ are chosen satisfying the following constraints:
\begin{equation}\label{def_coe}\begin{split}
&e^{\alpha_1\delta_0}\leq2, ~4m\leq M^{-1},~{m^2\over4}\alpha_1>|{p_x\over U}|_{L^\infty},\\
&C>\max\{M,{2\over m}|{p_x\over U}|_{L^\infty}\},~\beta\geq|B|_{L^\infty}+|{A\over1-\zeta}|_{L^\infty},
\end{split}\end{equation}
 for the positive constant $M$ given in \eqref{seq_ass}, and
\begin{equation}\label{def_coe1}\alpha~\geq~|B|_{L^\infty}+\max\Big\{\delta_0^{-1}\Big(|A\pd_\zeta\varphi|_{L^\infty}+C^2e^{2\beta T}|\pd_\zeta^2\varphi|_{L^\infty}\Big),~|{A\over1-\zeta}|_{L^\infty}\Big\},\end{equation}
by noting that ${A\over1-\zeta}$ is bounded, from the definition
\[A=-\zeta(1-\zeta){ U_t\over U}-(1-\zeta^2){ p_x\over U}.\]

With the above preparation, we have the following boundedness result on $W^n$.

\begin{lemma}\label{lem_seq1}
There exists $0<t_0\leq T$ such that for all $n$ and $t\in[0,t_0]$, the following estimate holds in $\Omega$:
\begin{equation}\label{seq_est1}
V_1(t,\zeta)\leq W^n(t,\xi,\eta,\zeta)\leq V_2(t,\zeta),
\end{equation}
where $V_1$ and $V_2$ are given in \eqref{def_V1-2}.
\end{lemma}

The proof is similar to that given in 
\cite[Lemma 4.3.2]{Ole} by using Lemma \ref{lem_seq} and the above construction of $(V_1, V_2)$, so we omit it here for brevity.

From the estimate \eqref{seq_est1}, we immediately have
\begin{corollary}\label{cor}
There exists a positive constant $M_1$ independent of $n$, such that when
$0\le t\leq t_0$,
\[M_1^{-1}(1-\zeta)\leq W^n(t,\xi,\eta,\zeta)\leq M_1(1-\zeta),\quad \forall n\geq0.\]
\end{corollary}
\bigskip

In the rest of this section, we will consider the problem only when
$0\le t\leq t_0.$

Now, we turn to estimate the first and second order derivatives of $W^n$. Let $V^n=W^ne^{\alpha\zeta}$, where $\alpha>0$ is a constant to be determined later. Then,
from the problem \eqref{eq_seq} of $W^n$, we know that
$V^n$ satisfies
\begin{equation}\label{V_eq}
\begin{cases}
L_n^0(V^n)+B^nV^n=0,\quad in~\Omega,\\
V^{n-1}(\pd_\zeta V^n-\alpha V^n)|_{\zeta=0}={p_x\over U},
\end{cases}\end{equation}
where
\begin{equation}\label{LV_def}
L_n^0(V^n)=\pd_t V^n+\zeta U(\pd_x+k\pd_y)V^n+A^n\pd_\zeta V^n-(W^{n-1})^2\pd_\zeta^2V^n
\end{equation}
with $A^n=A-2\alpha(W^{n-1})^2$, and
$B^n=B-\alpha A-\alpha^2(W^{n-1})^2$.

To study the first and second order derivatives of $V^n$ in $\Omega$,
similar to \cite{OA1}, introduce
 functions
\begin{equation}\label{phi}
\begin{split}\Phi_n&=(V^n_t)^2+(V^n_\xi)^2+(V^n_\eta)^2+(V^n_\zeta)^2+K_0+K_1\zeta\\
&=\sum_{|\gamma|=1}|\pd_\cT^\gamma V^n|^2+(V^n_\zeta)^2+K_0+K_1\zeta,\end{split}\end{equation}
and
\begin{equation}\label{psi}
\Psi_n=\sum_{|\gamma|=2}|\pd_\cT^\gamma V^n|^2+\sum_{|\gamma|=1}|\pd_\cT^\gamma V^n_{\zeta}|^2+(V^n_{\zeta\zeta})^2+N_0+N_1\zeta,
\end{equation}
where
$K_0,K_1,N_0$ and $N_1$ are positive constants to be specified later, and
\begin{equation}\label{def_tanop}
\pd_\cT^\gamma=\pd_t^{\gamma_1}\pd_\xi^{\gamma_2}\pd_\eta^{\gamma_3},\quad \gamma=(\gamma_1,\gamma_2,\gamma_3),\quad |\gamma|=\gamma_1+\gamma_2+\gamma_3,
\end{equation}
denotes the tangential differential operator to the boundaries $\{\zeta=0\}\cup\{\zeta=1\}$. For these two functions $\Phi_n$ and $\Psi_n$, the following results hold.

\begin{lemma}\label{lem_seq2} (1)
There are constants $K_0,K_1$ and $\alpha$ independent of $n$, such that for $n\geq1$,
\begin{equation}\label{seq_est_f1}
\begin{split}
&L_n^0(\Phi_n)+R^n\Phi_n\leq0,\quad in~\Omega,\\
&\pd_\zeta\Phi_n~\geq~\alpha\Phi_n-{\alpha\over2}\Phi_{n-1},
\quad on~\{\zeta=0\},
\end{split}\end{equation}
where $R^n$ is a function of $W^{n-1}$ and its first and second order derivatives.

(2)
There are two constants $N_0$ and $N_1$ depending only on the first order derivatives of $W^n$ and $W^{n-1}$, such that for $n\geq1$, one has
\begin{equation}\label{seq_est_s1}
\begin{split}
&L_n^0(\Psi_n)+C^n\Psi_n+N_2~\leq~0,\quad in~\Omega,\\
&\pd_\zeta\Psi_n\geq\alpha\Psi_n-{\alpha\over2}\Psi_{n-1},
\quad on~\{\zeta=0\},
\end{split}\end{equation}
where $N_2$ depends only on the first order derivatives of $W^n$ and $W^{n-1}$, while $C^n$ depends on $W^{n-1}$ and its first and second order derivatives.
\end{lemma}

\begin{proof}[\bf{Proof.}]
We will prove only the first part of this lemma, and the second result can be obtained similarly. The proof is divided in two steps.

 \underline{Step 1.}
 Let us consider $\pd_\zeta\Phi_n|_{\zeta=0}$ first. From the definition of $\Phi_n$, we have
\[\begin{split}\pd_\zeta\Phi_n&=2V^n_t V^n_{t\zeta}+2V^n_\xi V^n_{\xi\zeta}
+2V^n_\eta V^n_{\eta\zeta}+2V^n_\zeta V^n_{\zeta\zeta}+K_1\\
&=2\sum_{|\gamma|=1}\pd_\cT^\gamma V^n\cdot\pd_\cT^\gamma V^n_\zeta+2V^n_\zeta V^n_{\zeta\zeta}+K_1.
\end{split}\]
By Lemma \ref{lem_seq1}, the inequality $W^{n}|_{\zeta=0}=V^n|_{\zeta=0}\geq h_0>0$ holds for all $n$. Hence,  from the boundary condition given in \eqref{V_eq},
 we obtain on $\{\zeta=0\}$ that
\[\begin{split}
V^n_t V^n_{t\zeta}&=V^n_t\cdot\left[\alpha V_t^n+({p_x\over U})_t\cdot{1\over V^{n-1}}-{p_x\over U}\cdot{V^{n-1}_t\over (V^{n-1})^2}\right]\\
&\geq\alpha(V^n_t)^2-{\alpha\over4}(V^n_t)^2-{2\over\alpha}\left[({p_x\over U})_t\cdot{1\over V^{n-1}}\right]^2
-{2\over\alpha}\left[{p_x\over U}\cdot{1\over (V^{n-1})^2}\right]^2(V^{n-1}_t)^2,
\end{split}\]
which implies that
\[V_t^n V^n_{t\zeta}\geq{3\alpha\over4}(V^n_t)^2-{\alpha\over4}(V^{n-1}_t)^2
-K_2,\]
by choosing a constant $K_2\geq{2\over\alpha}\left|({p_x\over U})_t\cdot{1\over V^{n-1}|_{\zeta=0}}\right|_{L^\infty}^2$,
and
$\alpha>0$ large enough
such that
\begin{equation}\label{def_alpha}{2\over\alpha}\left[{p_x\over U}\cdot{1\over (V^{n-1})^2|_{\zeta=0}}\right]^2\leq{\alpha\over4}.\end{equation}
Obviously, the constants $K_2$ and $\alpha$ are independent of $n$.

Similarly, we can get the following two inequalities:
\[V_\xi^n V^n_{\xi\zeta}\geq{\alpha\over2}(V^n_\xi)^2-{\alpha\over4}(V^{n-1}_\xi)^2
-K_3,\]
and
\[V_\eta^n V^n_{\eta\zeta}\geq{\alpha\over2}(V^n_\eta)^2-{\alpha\over4}(V^{n-1}_\eta)^2
-K_3,\]
where $K_3$ is a constant independent of $n$, satisfying
 $$K_3\geq{1\over\alpha}\max\left\{\left|({p_x\over U})_\xi\cdot{1\over V^{n-1}|_{\zeta=0}}\right|_{L^\infty}^2,
\left|({p_x\over U})_\eta\cdot{1\over V^{n-1}|_{\zeta=0}}\right|_{L^\infty}^2\right\}.$$

On the other hand, from \eqref{V_eq},
we have that on $\{\zeta=0\}$,
\[\begin{split}V^n_\zeta V^n_{\zeta\zeta}&={1\over (W^{n-1})^2}V^n_\zeta\cdot[V^n_t+A^nV^n_\zeta+B^nV^n]\geq-{\alpha\over4}(V^n_t)^2-K_4,\end{split}\]
for a positive constant $K_4$ independent of $n$,
by using Lemma \ref{lem_seq1}
and
\begin{equation}\label{est_f_ze}
V^n_\zeta|_{\zeta=0}\leq K_5,
\end{equation}
 from the boundary condition  \eqref{V_eq}, with $K_5$ being a positive constant  independent of $n$.

Thus, we have that on $\{\zeta=0\}$,
\[\begin{split}\pd_\zeta\Phi_n~& \geq~\alpha[(V^n_t)^2+(V^n_\xi)^2+(V^n_\eta)^2]-
{\alpha\over2}[(V^{n-1}_t)^2+(V^{n-1}_\xi)^2+(V^{n-1}_\eta)^2]\\
&\qquad
-2(K_2+K_3+K_4)+K_1\\
& \geq~\alpha\Phi_n-{\alpha\over2}\Phi_{n-1}-K_6+K_1,
\end{split}\]
for a positive constant $K_6$ independent of $n$, which implies
the estimate \eqref{seq_est_f1} on
 $\{\zeta=0\}$  by choosing $K_1\geq K_6$.

 \underline{Step 2.}
 We turn to calculate $L_n^0(\Phi_n)$. Applying
the operator
\[2V^n_t\pd_t+2V^n_\xi\pd_\xi+2V^n_\eta\pd_\eta+2V^n_\zeta\pd_\zeta\]
to the equation $L^0_n(V^n)+B^nV^n=0,$
we have
\begin{equation}\label{2.68}
L_n^0(\Phi_n)+B^n\Phi_n-A^nK_1-B^n(K_0+K_1\zeta)+I_1+I_2+I_3=0,
\end{equation}
where
\[\begin{split}
&I_1
=2(W^{n-1})^2\cdot\left(\sum_{|\gamma|=1}|\pd_\cT^\gamma V^n_\zeta|^2+(V^n_{\zeta\zeta})^2\right),\\
&I_2
=-2V^n_{\zeta\zeta}\cdot\Big\{\sum_{|\gamma|=1}\pd_\cT^\gamma V^n\cdot\pd_\cT^\gamma [(W^{n-1})^2]+V^n_\zeta\cdot\pd_\zeta[(W^{n-1})^2]\Big\},\\
&I_3=2V^n_\xi\cdot\big(\zeta\sum_{|\gamma|=1}\pd_\cT^\gamma V^n \pd_\cT^\gamma U+V_\zeta^n U\big)
+2V^n_\eta\cdot\big(\zeta\sum_{|\gamma|=1}\pd_\cT^\gamma V^n\pd_\cT^\gamma(kU)+V_\zeta^n kU\big)\\
&\qquad
+2V^n_\zeta\cdot\big(\sum_{|\gamma|=1}\pd_\cT^\gamma V^n \pd_\cT^\gamma A^n+V_\zeta^n A_\zeta^n\big)
+2V^n\cdot\big(\sum_{|\gamma|=1}\pd_\cT^\gamma V^n \pd_\cT^\gamma B^n+V_\zeta^n B_\zeta^n\big).
\end{split}\]

Obviously, one has
\begin{equation}\label{I-2}
\begin{split}I_2\geq &-R_1\left(\sum_{|\gamma|=1}|\pd^\gamma_\cT V^n|^2+(V^n_\zeta)^2\right)
-{(V^n_{\zeta\zeta})^2\over R_1}\Big\{\sum_{|\gamma|=1}\Big|\pd_\cT^\gamma[(W^{n-1})^2]\Big|^2
+\Big|\pd_\zeta\big[(W^{n-1})^2\big]\Big|^2\Big\},\end{split}\end{equation}
where $R_1$ is a positive constant, and
\[I_3\geq-R_2\left(\sum_{|\gamma|=1}|\pd_\cT^\gamma V^n|^2+(V^n_\zeta)^2\right)-K_7,\]
with a constant $R_2$ depending on the bound of the first order derivatives of $W^{n-1}$ and
a constant $K_7$  independent of $n$.

To control the second term on the right hand side of the above inequality \eqref{I-2} of $I_2$,
we use the fact that the following inequality holds for an abitrary non-negative function $q(x)$
possessing
 bounded second derivatives for all $x$,
\begin{equation}\label{ineq_de}
(q_x)^2\leq2\left(\max|q_{xx}|\right)q.\end{equation}
The function $(W^{n-1})^2$ can be extended to  the whole space so that it is still
 non-negative, bounded and the magnitudes of its second order derivatives do
 not exceed the corresponding bound of the original function.
Hence, by using \eqref{ineq_de} we get
\[{(V^n_{\zeta\zeta})^2\over R_1}\left\{\sum_{|\gamma|=1}\Big|\pd_\cT^\gamma[(W^{n-1})^2]\Big|^2+\Big|\pd_\zeta\big[(W^{n-1})^2\big]\Big|^2\right\}\leq (W^{n-1})^2\cdot(V^n_{\zeta\zeta})^2,\]
when $R_1$ is sufficiently large and depends on the second order derivatives of $W^{n-1}$.

Therefore, from \eqref{2.68} we obtain
\[L_n^0(\Phi_n)+(B^n-R_1-R_2)\Phi_n-(B^n-R_1-R_2)(K_0+K_1\zeta)-A^nK_1\leq0,\]
which implies
\[L_n^0(\Phi_n)+R^n\Phi_n~\leq~0, \qquad {\rm in}~ \Omega,\]
for a function $R^n$ depending on $W^{n-1}$ and its first and second order derivatives,
by  choosing a suitable constant $K_0$.
\end{proof}

Next, we have the boundedness of the first and second order derivatives of
 $W^n$ stated in the following theorem.

\begin{theorem}\label{thm_seq}
  Suppose that the data in the problem \eqref{eq_tr} satisfies Assumption \ref{ass_1},
then there exists a $0<t_1\leq t_0$ such that the first and second order derivatives of the solution $W^n$ to \eqref{eq_seq} are bounded,
uniformly in $n$,
in $\Omega$ for $0\leq t\leq t_1$.
\end{theorem}

\begin{proof}[\bf{Proof.}]
From the definitions of $\Phi_n$ and $\Psi_n$ given in \eqref{phi} and \eqref{psi} respectively, it suffices to
prove that there exist constants $M_1,M_2$ and $t_1>0$, such that $\Phi_n\leq M_1$ and $\Psi_n\leq M_2$ hold for all $0\le t\leq t_1$
by induction on $n$.

The case of $n=0$ follows immediately
by noting that $W^0$ can be
chosen satisfying the requirement
 for all $t\leq T$. Assume that  $\Phi_i\leq M_1$ and $\Psi_i\leq M_2$  hold for all $0\le i\leq n-1$ when
$0\le t\leq t_1$, with
$t_1$ to be determined later. Denote by $\Omega_1=\Omega\cap \{t \leq t_1\}$.

Letting
$$\Phi_n^1=\Phi_ne^{-\gamma t},$$
from Lemma \ref{lem_seq2} we have that in $\Omega$, $\Phi_n^1\geq1$,
\begin{equation}\label{2.80}
L_n^0{\Phi_n^1}+(R^n+\gamma)\Phi_n^1\leq0,\end{equation}
and
\begin{equation}\label{bd_phi}
\pd_\zeta\Phi_n^1\geq\alpha\Phi_n^1-{\alpha\over2}\Phi_{n-1}^1,\quad{\rm on} ~\{\zeta=0\}.
\end{equation}
We choose $\gamma$, depending only on $M_1$ and $M_2$, such that $R^n+\gamma>0$ in $\Omega_1$. Then, by the maximum principle, from the equation \eqref{2.80} the function $\Phi_n^1$ does not attain its maximum value within $\Omega_1$, nor on $\{t=t_1\}\cup\{\zeta=1\}$
(since $W^{n-1}|_{\zeta=1}=0$).

If $\Phi_n^1$ attains its maximum at $\{t=0\}$, then we have
\[\Phi_n^1\leq\max\{\Phi_n|_{t=0}\}\leq K_8,\]
where $K_8$ is independent of $n$ and is determined by the parameters $k,U,A,B,W_0$ of the problem \eqref{eq_seq}, by using $V^n_t|_{t=0}=e^{\alpha\zeta}W^n_t|_{t=0}$ and
\[\begin{split}
W^n_t|_{t=0}=
(W_0)^2\pd_\zeta^2W_0-\zeta U|_{t=0}\cdot(\pd_\xi+k\pd_y)W_0-A^n|_{t=0}\cdot\pd_\zeta W_0-B^n|_{t=0}\cdot W_0.
\end{split}\]

If $\Phi_n^1$ attains its maximum
at a point $P$ on the boundary $\{(\xi,\eta)\in\pd D\}$, then at this point, we have that $\pd_t\Phi_n^1=\pd_\zeta \Phi_n^1=0,~\pd_\zeta^2 \Phi_n^1\leq0$ and
\[\pd_\tau\Phi_n^1~=~0,~\pd_n\Phi_n^1~\geq~0.\]
 By using
\[(\pd_\xi+k\pd_\eta)\Phi_n^1~=~k_\tau\pd_\tau\Phi_n^1+k_n\pd_n\Phi_n^1,\]
and $k_n<0$ on the boundary $\Gamma_-$,
from \eqref{2.80}
we get $P\in\Gamma_-$. On the other hand, from the problem \eqref{eq_seq}, we have
\begin{equation}\label{bd_g}
\pd_\tau W^n|_{\Gamma_-}~=~\pd_\tau W_1,
\end{equation}
and
\begin{equation}\label{bd_lateral}\left[\zeta U(\pd_\xi+k\pd_\eta)W^n\right]|_{\Gamma_-}=\left[(W_1)^2\pd_\zeta^2W_1-\pd_\tau W_1-A\pd_\zeta W_1-BW_1\right]|_{\Gamma_-}.\end{equation}
Hence, from the assumption of compatibility conditions,
we know that $\pd_n W^n=f_1/(\zeta U k_n)$ on $\Gamma_-$, with $f_1$  given in \eqref{def_w1}.
Thus, we know that $\pd_\tau W^n|_{\Gamma_-}$ and $\pd_n W^n|_{\Gamma_-}$ are bounded,
which implies that $\pd_\xi W^n$ and $\pd_\eta W^n$ are also bounded on the boundary $\Gamma_-$.
Therefore, we have
\[\Phi_n^1\leq\max\{e^{-\gamma t}\Phi_n|_{\Gamma_-}\}\leq K_9,\]
where $K_9$ is independent of $n$ and is determined by the parameters $k,U,A,B$ and $W_1$ of the problem \eqref{eq_seq}.

Finally, if $\Phi_n^1$ attains its maximum
at $\{\zeta=0\}$, then at this point, we have $\pd_\zeta\Phi_n^1\leq0$, and from \eqref{bd_phi} it follows that $\Phi_n^1\leq{1\over2}\Phi_{n-1}^1$, which implies that
\[\Phi_n^1\leq\max\{\Phi_n^1|_{\zeta=0}\}\leq {1\over2}\max\{\Phi_{n-1}^1\}\leq {M_1\over2},\]
by the induction assumption.

In conclusion, we obtain
\[\Phi_n^1\leq\max\{K_8,K_9,{M_1\over2}\},\quad in~\Omega_1,\]
which implies that
\[\Phi_n\leq\max\{K_8,K_9,{M_1\over2}\}e^{\gamma t},\quad in~\Omega_1.\]
Let $t_2\leq t_0$ be such that $e^{\gamma t_2}\le2$, and set $M_1=2\max\{K_8,K_9\}$. Obviously, $t_2$ and $M_1$ are independent of $n$. Then, it follows that $\Phi_n\leq M_1$ for $t\leq t_2$.

Similarly, we can obtain that $\Psi_n\leq M_2$ when $t\leq t_3$ for some $t_3$, where the choice of $t_3$ also depends only on the constants $M_1$ and $M_2$ given by the parameters $k,U,A,B,W_0,W_1$ of the problem \eqref{eq_tr}.

It follows that $\Phi_n\leq M_1$ and $\Psi_n\leq M_2$ for all $n$ when $t\leq t_1\triangleq \min\{t_2,t_3\}$, from which we obtain the boundedness of the first and second order derivatives of $W^n$.
\end{proof}

We can now prove the following existence result.

\begin{theorem}\label{thm_tr}
Suppose that the data in the problem \eqref{eq_seq} satisfies Assumption \ref{ass_1}, and let $\Omega_1=\Omega\bigcap\{t\leq t_1\}$ with $t_1$ being given in Theorem \ref{thm_seq}.
Then  the problem \eqref{eq_tr} has a unique solution $W$ in $\Omega_1$ satisfying that $W>0$ in $\Omega_1$, $W$ is bounded and continuous on $\overline{\Omega}_1$, and its first order derivatives and $W_{\zeta\zeta}$ are continuous and bounded in $\Omega_1$.
Moreover, we have the estimate
\begin{equation}\label{2.86}
M_1^{-1}(1-\zeta)\le W(t,\xi, \eta,\zeta)\le M_1(1-\zeta), \quad \forall (t,\xi, \eta,\zeta)\in\Omega_1
\end{equation}
for a positive constant $M_1>0$.
\end{theorem}

\begin{proof}[\bf{Proof.}]
First, we prove the existence of a solution $W$ to the problem \eqref{eq_tr}. In Theorem \ref{thm_seq}, we have shown that there exists a $t_1>0$ such that the first and second order derivatives of $W^n$ to the problem $\eqref{eq_seq}$ in $\Omega_1$ are bounded uniformly in $n$.  We are going to prove that $W^n$ converges uniformly in $\Omega_1$.

Letting $V^n=W^n-W^{n-1},~n\geq1$,  from \eqref{eq_seq} we know that for all $n\geq2$, $V^n$ satisfies the following problem in $\Omega_1$:
\begin{equation}\label{eq_dif}
\begin{cases}
\partial_t V^n+\zeta U(\partial_\xi +k\partial_\eta)V^n+A\pd_\zeta V^n+BV^n-(W^{n-1})^2\pd_\zeta^2 V^n\\
\qquad\qquad-\pd^2_\zeta W^{n-1}(W^{n-1}+W^{n-2})V^{n-1}=0,\\
\left(W^{n-1}\pd_\zeta V^n+\pd_\zeta W^{n-1} V^{n-1}\right)|_{\zeta=0}=0,
\quad V^n|_{\Gamma_-}= 0,\\
V^n|_{t=0}= 0.
\end{cases}\end{equation}
Moreover, from Lemma \ref{lem_seq} we have $V^n|_{\zeta=1}=0$.

Set $V^n_1=V^ne^{\alpha t+\beta\zeta}$. From \eqref{eq_dif} it follows that in $\Omega_1$,
\begin{equation}\label{2.87}
\begin{split}
&\partial_t V_1^n+\zeta U(\partial_\xi +k\partial_\eta)V_1^n+\left[A+2\beta(W^{n-1})^2\right]\pd_\zeta V^n-(W^{n-1})^2\pd_\zeta^2 V_1^n\\
&=\pd^2_\zeta W^{n-1}(W^{n-1}+W^{n-2})V^{n-1}_1+\left[\alpha-B+\beta A-\beta^2(W^{n-1})^2\right]V^n_1,
\end{split}\end{equation}
and on the boundary,
\begin{equation}\label{2.88}
\begin{cases}
V_1^n|_{t=0}=V_1^n|_{\zeta=1}=V^n_1|_{\Gamma_-}=0,\\[2mm]
W^{n-1}\pd_\zeta V_1^n|_{\zeta=0}=\left(\beta W^{n-1} V_1^n-\pd_\zeta W^{n-1} V^{n-1}_1\right)|_{\zeta=0}.
\end{cases}\end{equation}

By using Theorem \ref{thm_seq} and $W^{n-1}|_{\zeta=0}\geq h_0>0$, we choose the constant $\beta>0$ such that
 when $\zeta=0$,
\[\max\{\left|\pd_\zeta W^{n-1}\right|\}< q\beta \min\{W^{n-1}\},\]
for a positive constant $q<1$. Moreover, we choose the constant $\alpha<0$ such that in $\Omega_1$,
\[\max\{\left|\pd^2_\zeta W^{n-1}(W^{n-1}+W^{n-2})\right|\}< q\left(-\alpha-\max\{\left|B-\beta A+\beta^2(W^{n-1})^2\right|\}\right).\]

Hence, for the problem \eqref{2.87}-\eqref{2.88},
if $|V^n_1|$ attains its maximum at some interior or boundary point of  $\Omega_1$, we always have  \[\max\{\left|V_1^n\right|\}\le q\max\{\left|V_1^{n-1}\right|\},\]
which implies that the series $\sum_{n\geq1}V_1^n$ converges uniformly. It follows that there exists a function $W$ such that
\[W^n\rightarrow W\quad {\rm uniformly~in~} \Omega_1,\quad {\rm as}~n\rightarrow+\infty. \]
Meanwhile, we have $W|_{\zeta=1}=0$,
and satisfies the estimate \eqref{2.86}
by using Corollary \ref{cor}.

By using the inequality \eqref{ineq_de},
and  the uniform boundedness of $W^n$ and its first and second order derivatives, we get the uniform convergence of the first order derivatives of $W^n$ when $n\to +\infty$.

 Next, from the problem \eqref{eq_seq} of $W^n$, we know that for an arbitrary $\epsilon>0$ and when $\zeta<1-\epsilon$, $\pd_\zeta^2 W^n$ also converges uniformly as $n\to +\infty$.
Letting $n\to +\infty$ in \eqref{eq_seq} , it follows that $W$ satisfies the problem \eqref{eq_tr} in $\Omega_1$.

Now, we show the uniqueness of the solution $W$ to the problem \eqref{eq_tr}. Suppose that there are two solutions $W$ and $W'$ to \eqref{eq_tr}. Setting $V=W-W'$, then $V$ satisfies the following problem in $\Omega_1$:
\begin{equation*}\begin{cases}
\partial_t V+\zeta U(\partial_\xi +k\partial_\eta)V+A\pd_\zeta V+BV-W^2\pd_\zeta^2 V\\
\qquad\qquad-\pd^2_\zeta W'(W+W')V=0,\\
\left(W\pd_\zeta V+\pd_\zeta W'~V\right)|_{\zeta=0}=0,
~V|_{\zeta=1}~=~V|_{\Gamma_-}= 0,\\
V|_{t=0}= 0.
\end{cases}\end{equation*}
Consider the function $V_1~\triangleq~V e^{-\alpha_1 t+\beta_1\zeta}$ with $\alpha_1$ and $\beta_1$ being positive constants to be specified later. Then, we have
\begin{equation}\label{2.89}
\begin{cases}
\partial_t V_1+\zeta U(\partial_\xi +k\partial_\eta)V_1+\left(A+2\beta_1W^2\right)\pd_\zeta V_1-W^2\pd_\zeta^2 V_1\\
\qquad+\left[\alpha_1+B-\beta_1A+\beta_1^2W^2-\pd^2_\zeta W'(W+W')\right]V_1=0,\\
\left[W\pd_\zeta V_1+(\pd_\zeta W'-\beta_1W)V_1\right]|_{\zeta=0}=0,
\quad V_1|_{\zeta=1}~=~V_1|_{\Gamma_-}= 0,\\
V_1|_{t=0}= 0.
\end{cases}\end{equation}
If we choose $\alpha_1$ and $\beta_1$ sufficiently large such that
\[\alpha_1+B-\beta_1A+\beta_1^2W^2-\pd^2_\zeta W'(W+W')>0,\quad \pd_\zeta W'-\beta_1W<0,\]
then, for the problem \eqref{2.89},
$|V_1|$ does not attain its positive maximum at the interior and boundary points of $\Omega_1$. Consequently, $V_1~\equiv~0$, which yields  the uniqueness of the solution to the problem \eqref{eq_tr}.
\end{proof}

\subsection{Classical solution of the Prandtl equations}

Now, we return to the original problem \eqref{eqk3} of the Prandtl equations for $(u,w)$.
Assume that $p_x,U,u_0$ and $u_1$ are smooth and satisfy compatibility conditions such that the data in the problem
\eqref{eq_tr} after the Crocco transformation satisfy Assumption \ref{ass_1} given in Section 2.1.

Denote by
\[u_0^i(x,y,z)=\pd_t^iu|_{t=0},\quad w_0^i(x,y,z)=\pd_t^iw|_{t=0},\quad 0\leq i\leq 4,\]
and
\[u_1^i(t,x,y,z)=\pd_n^iu|_{\pd Q_T^-},\quad w_1^j(t,x,y,z)=\pd_n^jw|_{\pd Q_T^-},\quad 0\leq i\leq 4,~0\leq j\leq 3.\]
We now calculate these functions in terms of the initial and boundary data given in \eqref{eqk3}.

From the problem \eqref{eqk3}, obviously we have for $0\leq i\leq 4$,
\[w_0^i(x,y,z)=-\int_0^z\left[\pd_x u_0^i(x,y,\tilde z)+\pd_y(ku_0^i)(x,y,\tilde z)\right]d\tilde z,\]
and then for $0\leq j\leq3$,
\begin{equation}\label{def_u0}\begin{split}
u_0^{j+1}(x,y,z)=-(\pd_t^jp_x)|_{t=0}+\pd_z^2u_0^j-\sum_{k=0}^jC_j^k\left[
u_0^{j-k}\cdot(\pd_x+k\pd_y)u_0^k+w_0^{j-k}\cdot\pd_z u_0^k\right].
\end{split}\end{equation}

Next, from the divergence-free condition given in  \eqref{eqk3}, we have
\begin{equation}\label{w-1-0}
\begin{split}
w_1^0(t,x,y,z)&=-\int_0^z\left[\pd_xu+\pd_y(ku)\right]|_{\pd Q_T^-}d\tilde z\\
&\triangleq -g_0-k_n\int_0^zu_1^1(t,x,y,\tilde z)d\tilde z,
\end{split}\end{equation}
where $k_n=\left(1,k(x,y)\right)\cdot\vec{n}(x,y)$,
and
$$g_0(t,x,y,z)=\int_0^z\left(k_\tau\pd_\tau u_1+\pd_y k|_{\gamma_-}\cdot u_1\right)(t,x,y,\tilde z)d\tilde z,$$
with $k_\tau=\left(1,k(x,y)\right)\cdot\vec{\tau}(x,y)$. 

On the other hand, from
 the first equation of \eqref{eqk3}, we have that on the boundary $\pd Q_T^-$,
\[
k_nu_1\cdot u_1^1+w_1^0\cdot\pd_z u_1=-p_x+\pd_z^2 u_1-\pd_t u_1-k_\tau u_1\cdot\pd_\tau u_1,
\]
which implies that by using \eqref{w-1-0},
\begin{equation}\label{def_u1}
\begin{split}&k_n\left(u_1\cdot u_1^1-\pd_z u_1\cdot\int_0^z u_1^1d\tilde z\right)\\
&=-p_x+\pd_z^2 u_1-\pd_t u_1-k_\tau u_1\cdot\pd_\tau u_1+g_0\cdot\pd_z u_1~\triangleq~-f_1.
\end{split}\end{equation}

From \eqref{def_u1}, it follows that
\begin{equation}\label{def_u1-1}
k_n(u_1)^2\cdot\pd_z\left({\int_0^z~u_1^1~d\tilde z\over u_1}\right)=f_1,\end{equation}
by using that
\[\lim\limits_{z\rightarrow0^+}{\int_0^z~u_1^1~d\tilde z\over u_1}=0, 
\]
as a simple consequence from $\pd_z u_1>0$ and the compatibility conditions of $u_1$, $u_1^1|_{z=0}=0$.
Thus, from \eqref{def_u1-1} we deduce
\begin{equation}\label{u-1}
u_1^1={f_1\over k_n u_1}+\pd_z u_1\cdot\int_0^z{f_1\over k_n\cdot(u_1)^2}d\tilde z.
\end{equation}

In the same way as from \eqref{w-1-0} to \eqref{u-1},
we can compute $w_1^j$ ($1\leq j\leq3$) and $u_1^i$ ($2\leq i\leq4$), and there are smooth
function $g_j$ ($1\leq j\leq3$) and $f_i$ ($2\leq i\leq4$),
\begin{equation}\label{def_ui}
g_j=g_j(u_1^0,\cdots,u_1^{j}),~f_i=f_i\left(u_1,\cdots,u_1^{i-1};w_0^0,\cdots,w_1^{i-2}\right),
\end{equation}
such that
\[w_1^j(t,x,y,z)=-g_j-k_n\int_0^z~u_1^{j+1}(t,x,y,z)~d\tilde z,\]
and
\[u_1^i={f_i\over k_n u_1}+\pd_z u_1\cdot\int_0^z~{f_i\over k_n\cdot(u_1)^2}~d\tilde z.\]

Corresponding to Assumption \ref{ass_1},
we give the following assumption about the compatibility conditions of the problem \eqref{eqk3},

\begin{assumption}\label{ass_com}
Assume that for the problem \eqref{eqk3},
 $$k\in C^{10}(D),\quad (U,p)\in C^{10}\big((0,T]\times D\big),$$
and the initial-boundary data
$$u_0\in C^{15}(D\times\R_z^+),\quad~u_1\in C^{15}(\pd Q_T^-),$$
such that the following properties hold:
\begin{enumerate}
\item
$\lim\limits_{z\to+\infty}
u_0(x,y,z)=U(0,x,y),\quad
\lim\limits_{z\to+\infty}
u_1(t,x,y,z)=U(t,x,y)
\quad {\rm for~all~}(t,x,y)\in (0,T]\times\gamma_-$;

\item $\pd_z u_0>0,\pd_z u_1>0$ for all $z\geq0$, and there is constant $C_0>0$ such that
    \[
\qquad\quad C_0^{-1}\Big(U(0,x,y)-u_0(x,y,z)\Big)~\leq~\pd_z u_0(x,y,z)~\leq~ C_0\Big(U(0,x,y)-u_0(x,y,z)\Big),
\]
and
\[\begin{split}
& C_0^{-1}\Big(U(t,x,y)-u_1(t,x,y,z)\Big)~\leq~ \pd_z u_1(x,y,z)\\
&\qquad\qquad\qquad\qquad\qquad
~\leq~ C_0\Big(U(t,x,y)-u_1(t,x,y,z)\Big)\quad {\rm in}~(0,T]\times\gamma_- ;
\end{split}\]
\item $u_1^i\in C^7(\Gamma_-)~(1\leq i\leq 4)$, and the following compatibility conditions hold:
\[\begin{split}
\qquad\qquad{\bf{(i)}}&~\pd_z^{m+q}\pd_\tau^l\pd_n^ju_0^i|_{\pd Q_T^-}~=~\pd_z^{m+q}\pd_\tau^l\pd_t^iu_1^j|_{t=0},\quad{\rm for}~ i,j\leq4,~{\rm and}\\
&\qquad~m+l+i+j\leq5,~q=0,1;\\
\qquad\qquad{\bf{(ii)}}&~\pd_x^j\pd_y^m u_0^i|_{z=0}=\pd_t^q\pd_\tau^r u_1^s|_{z=0}=0,\quad{\rm for}  ~i+j+m\leq3,~s+q+r\leq3;\\
\qquad\qquad{\bf{(iii)}}&~\lim\limits_{z\rightarrow\infty}\pd_x^j\pd_y^m \pd_z u_0^i=\lim\limits_{z\rightarrow\infty}\pd_t^q\pd_\tau^r\pd_z u_1^s=0,\quad  ~i+j+m\leq3,~s+q+r\leq3.
\end{split}\]
\end{enumerate}
\end{assumption}

\begin{remark}
(1) It is easy to verify that Assumption \ref{ass_com} implies Assumption \ref{ass_1}.

(2) Under Assumption \ref{ass_com},
we know that
$u_1^i$ ($0\leq i\leq4$) are bounded continuous functions on the boundary $\pd Q_T^-$, and then the above computation of $u_1^i$ implies that we should have such boundary condition $u_1$ so that the functions $f_i$ ($1\leq i\leq 4$) defined in \eqref{def_u1} and \eqref{def_ui} satisfy
\begin{equation*}
f_i~=~\mathcal{O}(k_n\cdot z^2),\quad as~z~\rightarrow~0.
\end{equation*}
\end{remark}

Now, we give the following local well-posedness result of  the original
problem \eqref{eqk3}.
\begin{theorem}\label{thm_loc}
Suppose that the data in the problem \eqref{eqk3} satisfies Assumption \ref{ass_com}. Then,
there exists $0<T_0\leq T$ and a unique solution $(u,w)$ to the problem \eqref{eqk3} in the domain
 $Q_{T_0}$, satisfying
\begin{enumerate}
\item
$u>0$ when $z>0$, $\pd_z u>0$ when $z\geq0$,
\item
the derivatives $\pd_t u, \pd_x u,\pd_y u, \pd_z u, \pd_{tz}^2 u,\pd_{xz}^2 u,\pd_{yz}^2 u,\pd_{zz}^2 u$ and $\pd_z w$
are continuous and bounded in $Q_{T_0}$. Moreover, $\pd_z^2 u/ \pd_z u$ and $(\pd_z u\pd_z^3 u-(\pd_z^2 u)^2)/(\pd_z u)^3$ are continuous and bounded in $Q_{T_0}$.
\end{enumerate}
\end{theorem}

\begin{proof}[\bf{Proof.}] The proof is divided in two steps.

\underline{Step 1.}
 Let $T_0=t_1$ and $W$ be the solution to the problem \eqref{eq_tr}, where $t_1$ and $W$ are obtained in Theorem \ref{thm_tr}.
Define $u(t,x,y,z)$ by using the relation
\begin{equation}\label{tr_inverse}z=\int_0^{u/U}{ds\over W(t,x,y,s)}.\end{equation}
By using the continuity of $W$ in $\overline{\Omega}$ and $W(t,x,y,s)>0$ for $0\leq s<1$, $W=0$ at $s=1$, we obtain that $u(t,x,y,z)/U(t,x,y)$ is continuous in $\overline{Q_{T_0}}$,
\[u|_{z=0}=0,\quad \lim\limits_{z\rightarrow+\infty}u=U(t,x,y),\]
and
\(0<u(t,x,y,z)<U(t,x,y)\) as \(0<z<+\infty.\)
From \eqref{tr_inverse},
we have $\pd_z u/U= W(t,x,y,u/U)$, and then
 the conditions 
$$u|_{t=0}=u_0(x,y,z), \qquad u|_{\pd Q_{T_0}^-}=u_1(t,x,y,z)|_{\pd Q_{T_0}^-}$$
follow from $W_0=\frac{\pd_z u_0}{U}$ and $W_1=\frac{\pd_z u_1}{U}$, respectively.

The first equation given in the Prandtl equations \eqref{eqk3} leads to define
\begin{equation}\label{w_def}w={-\pd_t u-u(\pd_x+k\pd_y)u+\pd_z^2 u-p_x\over\pd_z u}.\end{equation}
Since $u|_{z=0}=0$, from
\eqref{w_def} it follows that
\[w|_{z=0}=\left({\pd_z^2 u-p_x\over \pd_z u}\right)\Big|_{z=0}=\left({UWW_\zeta-p_x\over UW}\right)\Big|_{\zeta=0}=0,\]
by using that
\begin{equation}\label{u_z}
\pd_z u=UW, \quad \pd_z^2 u=W_\zeta \pd_z u=UWW_\zeta,
\end{equation}
and the boundary condition $WW_\zeta|_{\zeta=0}={p_x\over U}$ given in \eqref{eq_tr}.

Moreover, from \eqref{u_z} we get that
\begin{equation}\label{tr_1}\begin{split}
&\pd_z^3 u= W_{\zeta\zeta}\frac{(\pd_z u)^2}{U}+W_\zeta \pd_z^2 u,\quad\pd_{zt}^2 u= U_tW+U\big(W_t+W_\zeta\pd_t ({u\over U})\big),\\
&\pd_{zx}^2 u=U_xW+U\big(W_\xi+W_\zeta\pd_x ({u\over U})\big),\quad\pd_{zy}^2 u=U_yW+U\big(W_\eta+W_\zeta\pd_y ({u\over U})\big),\\
 &\pd_t u=u{U_t\over U}+UW\int_0^{u/U}{W_t\over W^2}ds,\quad\pd_x u=u{U_x\over U}+UW\int_0^{u/U}{W_\xi\over W^2}ds,\\
 &\pd_y u=u{U_y\over U}+UW\int_0^{u/U}{W_\eta\over W^2}ds.
 \end{split}\end{equation}
So,
from the properties of $W$ given in Theorem \ref{thm_tr} and the above definition \eqref{tr_inverse} of $u$,  as in \cite{OA1} it is not difficult to obtain the continuity and boundedness of $u$ and its derivatives as stated in the theorem.

\underline{Step 2.}
We will show that $(u,w)$ given by \eqref{tr_inverse} and \eqref{w_def} satisfies the problem \eqref{eqk3}. The first equation in \eqref{eqk3} holds trivially.

To verify that $(u,w)$ satisfies the second equation in \eqref{eqk3}, by differentiating \eqref{w_def} with respect to $z$, it yields
\begin{equation}\label{w_z}\begin{split}
&\pd_z u~\pd_z w+\frac{\pd_z^2 u}{\pd_z u}\left[-\pd_t u-u(\pd_x+k\pd_y)u+\pd_z^2 u-p_x\right]\\
&=-\pd_{tz}^2u-u(\pd_x+k\pd_y)\pd_z u-\pd_zu~(\pd_x+k\pd_y)u+\pd_z^3 u.
\end{split}\end{equation}
Then, substituting \eqref{u_z} and \eqref{tr_1} into \eqref{w_z} yields that
\begin{equation*}\begin{split}
&UW\big[\pd_z w+(\pd_x+k\pd_y)u\big]+W_\zeta\left[-\pd_t u-u(\pd_x+k\pd_y)u+\pd_z^2 u-p_x\right]\\
&=-U_tW-U\big(W_t+W_\zeta\pd_t ({u\over U})\big)-u\big[U_xW+U\big(W_\xi+W_\zeta\pd_x ({u\over U})\big)\big]\\
&\quad-ku\big[U_yW+U\big(W_\eta+W_\zeta\pd_y ({u\over U})\big)\big]+W_{\zeta\zeta}\frac{(\pd_z u)^2}{U}+W_\zeta \pd_z^2 u,
\end{split}\end{equation*}
that is,
\begin{equation}\label{w-z}\begin{split}
&W\big[\pd_z w+(\pd_x+k\pd_y)u\big]-\frac{p_x}{U}W_\zeta=\frac{- U_t-u U_x-kuU_y}{U}W\\
&\qquad-W_t-uW_\xi-kuW_\eta+\frac{uU_t+u^2U_x+ku^2U_y}{U^2}W_\zeta
+W^2W_{\zeta\zeta}.
\end{split}\end{equation}
Combining \eqref{w-z} with the equation of \eqref{eq_tr} for $W(\tau,\xi,\eta,u/U)$, it follows that
\begin{equation}\label{div}W\big[\pd_z w+(\pd_x+k\pd_y)u\big]=-\pd_y k~uW+\frac{u^2}{U^3}\big(U_t+UU_x+kUU_y+p_x\big).\end{equation}
Thus, from \eqref{div} and the Bernoulli law:
\(U_t+UU_x+kUU_y+p_x=0,\)
it follows that $W \big(\pd_z w+\pd_x u+\pd_y(ku)\big)=0$, which implies that by virtue of $W>0$,
\[\pd_x u+\pd_y(ku)+\pd_z w=0.\]
So, we obtain that $(u,w)$ satisfies the problem \eqref{eqk3}.

Uniqueness of the solution to the problem \eqref{eqk3} follows from the uniqueness of the solution to the problem \eqref{eq_tr} given in Theorem \ref{thm_tr}. Hence, we complete the proof of this theorem.
\end{proof}

\begin{remark}
As discussed in Section 1, from Theorem \ref{thm_loc} we immediately deduce the existence and uniqueness of the classical solution to the problem \eqref{pdtl-1} as claimed in Theorem \ref{1.1}.
\end{remark}


\section{Linear stability with a general perturbation}

 In this section, we will study the
stability of the classical solution with the special structure constructed in Section 2 for the problem \eqref{pdtl-1}
with respect to any three dimensional perturbation.

That is, let
\[\Big(u^s(t,
z,y,z), k(x,y)u^s(t,x,y,z), w^s(t,x,y,z)\Big)\]
be a classical solution to the problem \eqref{pdtl-1},
consider the following linearized problem of \eqref{pdtl-1} around
this solution profile in $Q_T=(0, T]\times Q$ with
$Q=D\times \R^+_z$:
\begin{equation}\label{leqp}
\begin{cases}
  \partial_t u+(u^s\partial_x+ku^s\partial_y+w^s\partial_z)u+
  (u\partial_x+v\partial_y+w\partial_z)u^s-\partial_z^2u=f_1,\\
\partial_t v+(u^s\partial_x+ku^s\partial_y+w^s\partial_z)v+
  (u\partial_x+v\partial_y+w\partial_z)(ku^s)-\partial_z^2v=f_2,\\
\partial_x u+\partial_y v+\partial_zw=0,\\
(u,v,w)|_{z=0}=0, \quad \lim\limits_{z\rightarrow +\infty}(u,v)=0,\quad
(u,v)|_{\pd Q_T^-}=(u_1,v_1)(t,x,y,z)|_{\pd Q_T^-},\\
(u,v)|_{t=0}=(u_0,v_0)(x,y,z),
\end{cases}\end{equation}
where $\pd Q_T^-=(0,T]\times \gamma_-\times \R^+_z$ with $\gamma_-$ being defined in  \eqref{pdtl-1}.

We will apply the energy method introduced
in \cite{AWXY} for the  two dimensional Prandtl equations and use the special structure of the problem \eqref{leqp} in the three dimensional setting. For this,  we firstly recall some weighted norms introduced in \cite{AWXY}. For any function $f(t,x,y,z)$ defined in $Q_T$, real numbers $\lambda,l>0$ and $j, j_1, j_2\in \N$, define the spaces
$L^2_{\lambda, l}(Q)$, $\cB_\lam^{j_1, j_2}(Q_T)$,  $\tilde \cB_\lam^{j_1, j_2}(Q_T)$, $\cA_l^j(Q_T)$ and $\cD_l^j(Q_T)$ with the corresponding norms,
$$\|f\|_{\lambda, l}=\left(\int_{Q} e^{-2\lambda
t}\langle z\rangle^{2l}|f|^2 dxdydz\right)^{\frac{1}{2}},
\quad \langle z\rangle=(1+z^2)^{\frac{1}{2}},$$
$$
\|f\|_{\cB_\lam^{j_1, j_2}}=\left(\sum\limits_{0\le m\le j_1, 0\le
q\le j_2} \|e^{-\lambda t}\wt^l\pd_\cT^m\pd_z^q f\|^2_{L^2(Q_T)}\right)^{\frac{1}{2}},
$$
$$
\|f\|_{\tilde \cB_\lam^{j_1, j_2}}=\left(\sum\limits_{0\le m\le j_1,
0\le q\le j_2} \|e^{-\lambda t}\wt^l\pd_\cT^m\pd_z^q
f\|^2_{L^\infty(0, T; L^2(Q))}\right)^{\frac{1}{2}},$$
$$\|f\|_{\cA_l^j(Q_T)}=\left(\sum\limits_{j_1+\left[\frac{j_2+1}{2}\right]\le
j} \|\wt^l\pd_\cT^{j_1}\pd_z^{j_2} f\|^2_{L^2(
Q_T)}\right)^{\frac{1}{2}},
$$
and
\[\|f\|_{\cD_l^j}=\sum\limits_{j_1+[{j_2+1\over2}]\leq j}\|\wt^l\pd_\cT^{j_1}\pd_z^{j_2} f\|_{L_z^\infty(L^2_{t,x,y})},\]
with
$$\pd_\cT^j=\sum\limits_{|\beta|\le
j}\pd_t^{\beta_1}\pd_x^{\beta_2}\pd_y^{\beta_3},\quad |\beta|=\beta_1+\beta_2+\beta_3,$$
being the tangential derivatives along with the physical boundary $\{z=0\}$.

Also, we use the following notations:
\begin{equation}\label{def-linsta}
\eta=\frac{\partial_z^2 u^s}{\partial_z u^s}, \quad
\zeta=\frac{(\partial_t+u^s\partial_x+k
u^s\partial_y+w^s\partial_z-\partial_z^2)\partial_z
u^s}{\partial_z u^s},\quad \tilde
f=\frac{f_1}{\partial_z u^s}.
\end{equation}
To study the problem \eqref{leqp}, let us first impose the following assumption.

\begin{assumption}\label{ass-linsta}
For fixed integer $j\geq4$ and real number $l>{1\over2}$, assume that
\begin{enumerate}
\item the background state $(u^s,ku^s,w^s)$ satisfying
$\pd_z u^s>0, u^s, w^s\in\cD_l^j(Q_T)$ and
functions $k\in C^{j+1}(D)$, such that functions $\eta,\zeta$ given in \eqref{def-linsta} satisfy  $\eta\in\cD_0^{j}(Q_T)$ and $\zeta\in\cA_l^{j}(Q_T)$;
\item $(f_1,f_2)\in\cA_l^j(Q_T)$, and $(\frac{f_1}{\pd_z u^s},\frac{f_2}{\pd_z u^s})\in\cA_l^{j-1}(Q_T)$;
\item the initial boundary data satisfy $(u_0,v_0)\in H^{2j}(Q),~(u_1,v_1)\in H^{2j}(\pd Q_T^-)$, and the compatibility conditions of \eqref{leqp}
 up to the  $(j-1)$-th order. Moreover,  the following estimates hold:
\begin{equation}\label{est-ass1}
\begin{split}
&\sum_{i+j_1+j_2\leq j}\Big(\|\wt^l\pd_x^{j_1}\pd_y^{j_2} u_0^i\|_{L^2(Q)}+\|\wt^l\pd_x^{j_1}\pd_y^{j_2} v_0^i\|_{L^2(Q)}\\
&\qquad\qquad+\|\wt^l\pd_t^{j_1}\pd_\tau^{j_2} u_1^i\|_{L^2(\pd  Q_T^-)}+\|\wt^l\pd_t^{j_1}\pd_\tau^{j_2} v_1^i\|_{L^2(\pd Q_T^-)}\Big)\leq {\mathcal M}_0,
\end{split}\end{equation}
\begin{equation}\label{est-ass2}
\begin{split}
&\sum_{i+j_1+j_2\leq j}\Big(\|\wt^l{\pd_x^{j_1}\pd_y^{j_2} u_0^i\over\pd_z u^s(0,\cdot)}\|_{L^2(Q)}+\|\wt^l{\pd_t^{j_1}\pd_\tau^{j_2} u_1^i\over\pd_z u^s}\|_{L^2(\pd  Q_T^-)}\\
&\qquad\qquad+\|\wt^l{\pd_x^{j_1}\pd_y^{j_2} v_0^i\over\pd_z u^s(0,\cdot)}\|_{L^2(Q)}+\|\wt^l{\pd_t^{j_1}\pd_\tau^{j_2} v_1^i\over\pd_z u^s}\|_{L^2(\pd  Q_T^-)}\Big)\leq {\mathcal M}_1,
\end{split}\end{equation}
for two positive constants  ${\mathcal M}_0$ and ${\mathcal M}_1$,
with
\[u_0^i(x,y,z)=\pd_t^iu|_{t=0},\quad u_1^i(t,x,y,z)=\pd_n^iu|_{\pd Q_T^-},\]
\[v_0^i(x,y,z)=\pd_t^iv|_{t=0},\quad v_1^i(t,x,y,z)=\pd_n^iv|_{\pd Q_T^-},\]
for $i\le j$. Here,  $\pd_n=\vec{n}\cdot\nabla_{(x,y)}$ and $\pd_\tau=\vec{\tau}\cdot\nabla_{(x,y)}$.
\end{enumerate}
\end{assumption}

The following result shows that the classical solution to the nonlinear Prandtl equations \eqref{pdtl-1} obtained in Section 2 is linearly stable with respect to any three-dimensional perturbation of initial and boundary data without the special structural constraint.

\begin{theorem}\label{thm-linsta}
 Under Assumption \ref{ass-linsta}, the problem \eqref{leqp}
 has a unique solution $(u,v,w)$ satisfying $(u,v)\in \cA_l^{j-1}(Q_T)$ and $w\in\cD_0^{j-2}(Q_T)$, moreover, it is stable with respect to the initial-boundary data and the source terms in the sense that the following estimate holds, for a constant $C$ depending on the bounds of $k$ and the background state,
 \begin{equation}\label{est_or}\begin{split}
 &\|u\|_{\cA_l^{j-1}}+\|v\|_{\cA_l^{j-1}}+\|w\|_{\cD_0^{j-2}}\\
 &\leq C\Big(\mathcal{M}_0+\mathcal{M}_1+\|\frac{f_1}{\pd_z u^s}\|_{\cA_l^{j-1}}+\|\frac{f_2}{\pd_z u^s}\|_{\cA_l^{j-1}}+\|f_1\|_{\cA_l^j}+\|f_2\|_{\cA_l^{j}}\Big),
 \end{split}\end{equation}
  where $\mathcal{M}_0$ and $\mathcal{M}_1$ are bounds of initial-boundary data given in \eqref{est-ass1} and \eqref{est-ass2}.
\end{theorem}

To prove this theorem, by using the special structure of the problem \eqref{leqp}, we first introduce a new unknown function
\begin{equation}\label{def_tv}
\tilde v(t,x,y,z)=k(x,y)u(t,x,y,z)-v(t,x,y,z).
\end{equation}
By the relation $k_x+kk_y=0$, from \eqref{leqp} we know that $\tilde v(t,x,y,z)$ satisfies the following problem 
\begin{equation}\label{eq_tv}
\begin{cases}
\partial_t\tilde v+(u^s\partial_x+ku^s\partial_y+w^s\partial_z)\tilde
v+k_yu^s\tilde v-\partial_z^2\tilde v=kf_1-f_2,\quad in~ Q_T,\\
\tilde v|_{z=0}=0,\quad \lim\limits_{z\rightarrow+\infty}\tilde v=0,\quad
\tilde v|_{\pd Q_T^-}=(ku_1-v_1)(t,x,y,z),\\
\tilde v|_{t=0}=(ku_0-v_0)(x,y,z).
\end{cases}\end{equation}
And for the problem \eqref{eq_tv}, we have

\begin{lemma}\label{lem-linsta1}
Under Assumption \ref{ass-linsta}, the problem \eqref{eq_tv} has a unique smooth solution $\tilde{v}(t,x,y,z)$,
and there is a constant $C>0$ such that
\begin{equation}\label{est-linsta1}
\|\tv\|_{\cA_l^j}\leq C\Big({\mathcal M}_0+\|kf_1-f_2\|_{\cA_l^j}\Big),
\end{equation}
and
\begin{equation}\label{est-linsta1-1}
\|{\tv\over\pd_z u^s}\|_{\cA_l^{j-1}}\leq C\Big({\mathcal M}_1+\|{kf_1-f_2\over\pd_z u^s}\|_{\cA_l^{j-1}}\Big).
\end{equation}
\end{lemma}

\begin{proof}[\bf{Proof.}]
From Assumption \ref{ass-linsta}, we know that the compatibility conditions
of the problem \eqref{eq_tv} hold up to the $(j-1)$-th order. So, the main task is to prove \eqref{est-linsta1} and \eqref{est-linsta1-1} which can be obtained in the following four steps.

\underline{Step 1.} $L^2$-estimate of $\tv$.

Multiplying $\eqref{eq_tv}_1$ by $e^{-2\lambda t}\wt^{2l}\tilde v$ and
integrating over $Q$, we get
\begin{equation}\label{est_tv}
\begin{split}
&{d\over 2dt}\|\tv(t)\|_\lam^2+\lambda\|\tv(t)\|_\lam^2+\int_Q(u^s\partial_x+ku^s\partial_y+w^s\partial_z)\tv\cdot e^{-2\lambda t}\wt^{2l}\tilde v-\int_Q\pd_z^2\tv\cdot e^{-2\lambda t}\wt^{2l}\tilde v\\
&\leq
\|(kf_1-f_2)(t)\|_\lam\cdot\|\tv(t)\|_\lam+\|k_yu^s(t)\|_{L^\infty(Q)}\cdot\|\tv(t)\|_\lam^2.
\end{split}\end{equation}

Now, we estimate the last two terms on the left hand side of \eqref{est_tv}. First,
from the boundary condition given in \eqref{eq_tv} on $\pd Q_T^-$, it follows that
\begin{equation}\label{est_tv0}
\begin{split}
&\quad\int_Q(u^s\partial_x+ku^s\partial_y+w^s\partial_z)\tv\cdot e^{-2\lambda t}\wt^{2l}\tilde v\\
&={1\over2}\int_{\pd D}\int_0^\infty e^{-2\lambda t}\wt^{2l}u^s|\tv|^2(1,k)\cdot\vec{n}-l\int_Q e^{-2\lambda t}z\wt^{2l-2}w^s|\tv|^2\\
&\geq {1\over2}\int_{\gamma_-}\int_0^\infty e^{-2\lambda t}\wt^{2l}k_nu^s|ku_1-v_1|^2-l\|w^s(t)\|_{L^\infty(Q)}\|\tv(t)\|_\lam^2,
\end{split}\end{equation}
where the function $k_n=(1,k)\cdot\vec{n}$ is defined on the boundary $\pd D$.

By using the boundary condition
\(\tilde v|_{z=0}=0,\)
we have
\begin{equation}\label{est_tv0-1}\begin{split}
-\int_Q\pd_z^2\tv\cdot e^{-2\lambda t}\wt^{2l}\tilde v
&=\int_Q e^{-2\lambda t}\pd_z\tv\cdot(\wt^{2l}\pd_z\tv+2l\wt^{2l-2}z\tv)\\
&\geq\|\pd_z\tv(t)\|_\lam^2-2l\|\pd_z\tv\|_\lam\cdot\|\tv(t)\|_\lam\\
&\geq{1\over2}\|\pd_z\tv(t)\|_\lam^2-2l^2\|\tv(t)\|_\lam^2.
\end{split}\end{equation}

Plugging \eqref{est_tv0} and \eqref{est_tv0-1} into \eqref{est_tv},
and choosing $\lambda$ large enough such that
\[\lambda\geq 1+2l^2+2l\|w^s\|_{L^\infty}+2\|k_yu^s\|_{L^\infty},\]
we obtain
\begin{equation}\label{est_tv0-2}
\begin{split}
&\quad{d\over dt}\|\tv(t)\|_\lam^2+\lambda\|\tv(t)\|_\lam^2+\|\pd_z\tv(t)\|_\lam^2\\
&\leq
\|(kf_1-f_2)(t)\|_\lam^2-\int_{\gamma_-}\int_0^\infty e^{-2\lambda t}\wt^{2l}u^s|ku_1-v_1|^2(1,k)\cdot\vec{n}.
\end{split}\end{equation}
Integrating \eqref{est_tv0-2} over $(0,t),~t\in(0,T]$, we get
\begin{equation}\label{est_tv0-3}
\begin{split}
&\quad\|\tv\|_{\tilde\cB_\lam^{0,0}}^2+\lambda\|\tv\|_{\cB_\lam^{0,0}}^2+\|\tv\|_{\cB_\lam^{0,1}}^2\\
&\leq \|ku_0-v_0\|_{\lam}^2+
\|kf_1-f_2\|_{\cB_\lam^{0,0}}^2-\int_{\pd Q_T^-} e^{-2\lambda t}\wt^{2l}
|ku_1-v_1|^2(1,k)\cdot\vec{n}\\
&\leq \|kf_1-f_2\|_{\cB_\lam^{0,0}}^2+\|ku_0-v_0\|_{\lam}^2+\|u^sk_n\|_{L^\infty(\pd Q_T^-)}\cdot\|e^{-\lambda t}\wt^l(ku_1-v_1)\|_{L^2(\pd Q_T^-)}^2.
\end{split}\end{equation}

 \underline{Step 2.} Estimates of tangential derivatives $\pd_\cT^\beta\tv~(|\beta|\leq j)$.

Applying the operator $\pd_\cT^\beta~(|\beta|\leq j)$ to the equation
 $\eqref{eq_tv}_1$, multiplying the resulting equation by $e^{-2\lambda t}\wt^{2l}\pd_\cT^\beta\tilde v$
 and integrating over $Q$, we get
\begin{equation}\label{est_tv1}
\begin{split}
&\quad{d\over 2dt}\|\pd_\cT^\beta\tv(t)\|_\lam^2+\lambda\|\pd_\cT^\beta\tv(t)\|_\lam^2+\|\pd_z\pd_\cT^\beta\tv(t)\|_\lam^2+\sum_{i=1}^3I_i\\
&\leq 2l\|\pd_z\pd_\cT^\beta\tv(t)\|_\lam\cdot\|\pd_\cT^\beta\tv(t)\|_\lam+
\|\pd_\cT^\beta(kf_1-f_2)(t)\|_\lam\cdot\|\pd_\cT^\beta\tv(t)\|_\lam,
\end{split}\end{equation}
where
\[\begin{split}
&I_1=\int_Q(u^s\partial_x+ku^s\partial_y+w^s\partial_z)\pd_\cT^\beta\tv\cdot e^{-2\lambda t}\wt^{2l}\pd_\cT^\beta\tilde v,\\
&I_2=\int_Q\big[\pd_\cT^\beta,u^s\partial_x+ku^s\partial_y+w^s\partial_z\big]\tv\cdot e^{-2\lambda t}\wt^{2l}\pd_\cT^\beta\tilde v,\\
&I_3=\int_Q\pd_\cT^\beta(k_yu^s\tv)\cdot e^{-2\lambda t}\wt^{2l}\pd_\cT^\beta\tilde v,
\end{split}\]
with the notation $[\cdot,\cdot]$ denoting the commutator.

We estimate the terms $I_i~(i=1,2,3)$ given in \eqref{est_tv1}. Obviously, we have
\[\begin{split}
I_1&={1\over2}\int_{\pd D}\int_0^\infty e^{-2\lambda t}\wt^{2l}u^s|\pd_\cT^\beta\tv|^2(1,k)\cdot\vec{n}
-l\int_Q e^{-2\lambda t}z\wt^{2l-2}w^s|\pd_\cT^\beta\tv|^2\\
&\geq{1\over2}\int_{\gamma_-}\int_0^\infty e^{-2\lambda t}\wt^{2l}k_nu^s|\pd_\cT^\beta\tv|^2-l\|w^s(t)\|_{L^\infty(Q)}\|\pd_\cT^\beta\tv(t)\|_\lam^2,
\end{split}\]
which implies that
\begin{equation}\label{est_tv2}
-I_1\lesssim\sum_{j_1+j_2+j_3\leq|\beta|}\int_{\gamma_-}\int_0^\infty e^{-2\lambda t}\wt^{2l}
|\pd_t^{j_1}\pd_\tau^{j_2}\pd_n^{j_3}\tv(t)|^2+\|\pd_\cT^\beta\tv(t)\|_\lam^2,
\end{equation}
by using that
\[\begin{split}
\Big|\int_{\gamma_-}\int_0^\infty e^{-2\lambda t}\wt^{2l}k_nu^s|\pd_\cT^\beta\tv|^2\Big|
\lesssim\sum_{j_1+j_2+j_3\leq|\beta|}\int_{\gamma_-}\int_0^\infty e^{-2\lambda t}\wt^{2l}
|\pd_t^{j_1}\pd_\tau^{j_2}\pd_n^{j_3}\tv|^2.
\end{split}\]

Secondly, by using commutator estimates given in \cite{taylor}, we get
\[\begin{split}
|I_2|&\leq\|\pd_\cT^\beta\tv(t)\|_\lam\cdot\Big\|\Big[\pd_\cT^\beta,u^s\partial_x+ku^s\partial_y+w^s\partial_z\Big]\tv(t)\Big\|_\lam\\
&\lesssim\|\pd_\cT^\beta\tv(t)\|_\lam\cdot\Big[\big(\|u^s(t)\|_{L^\infty(Q)}+\|w^s(t)\|_{L^\infty(Q)}\big)\cdot\|\tv\|_{\cB_\lam^{|\beta|,1}(t)}\\
&\quad +\big(\|u^s\|_{\cD_0^{|\beta|}(t)}+\|w^s\|_{\cD_0^{|\beta|}(t)}\big)\cdot\|\tv\|_{\cB_\lam^{3,1}(t)}\Big],
\end{split}\]
with the notations $\|\cdot\|_{\cB_\lam^{j_1.j_2}(t)}$ and $\|\cdot\|_{\cD_l^{j}(t)}$  given by
\[\|f\|_{\cB_\lam^{j_1, j_2}(t)}=\left(\sum\limits_{0\le m\le j_1, 0\leq
q\leq j_2} \|e^{-\lambda t}\wt^l\pd_\cT^m\pd_z^q f(t)\|^2_{L^2(Q)}\right)^{\frac{1}{2}},
\]
and
\[\|f\|_{\cD_l^j(t)}=\sum\limits_{j_1+[{j_2+1\over2}]\leq j}\|\wt^l\pd_\cT^{j_1}\pd_z^{j_2} f(t)\|_{L_z^\infty(L^2_{x,y})}.\]
Thus, we have
\begin{equation}\label{est_tv3}
|I_2|\leq{1\over4}\big(\|\tv\|^2_{\cB_\lam^{|\beta|,1}(t)}+\|\tv\|^2_{\cB_\lam^{3,1}(t)}\big)+C\|\pd_\cT^\beta\tv(t)\|^2_\lam.
\end{equation}
Similarly, for the term $I_3$, we obtain
\begin{equation}\label{est_tv4}
\begin{split}
|I_3|&\leq\|\pd_\cT^\beta\tv(t)\|_\lam\cdot\|\pd_\cT(k_yu^s\tv)(t)\|_\lam\\
&\lesssim\|\pd_\cT^\beta\tv(t)\|_\lam\cdot\Big[\|k_yu^s(t)\|_{L^\infty(Q)}\|\tv\|_{\cB_\lam^{|\beta|,0}(t)}+\|k_yu^s\|_{\cD_0^{|\beta|}(t)}\|\tv\|_{\cB_\lam^{2,0}(t)}\Big].
\end{split}\end{equation}
Plugging \eqref{est_tv2}-\eqref{est_tv4} into \eqref{est_tv1}, taking
summation over all $|\beta|\leq j$ and choosing $\lambda$ large enough, we obtain that
\begin{equation}\label{est_tv5}
\begin{split}
&\quad{d\over dt}\|\tv\|_{\cB_\lam^{j,0}(t)}^2+\lambda\|\tv\|_{\cB_\lam^{j,0}(t)}^2+\|\tv\|_{\cB_\lam^{j,1}(t)}^2\\
&\lesssim
\|kf_1-f_2\|_{\cB_\lam^{j,0}(t)}^2+\|\tv\|_{\cB_\lam^{3,1}(t)}^2+\sum_{j_1+j_2+j_3\leq j}\int_{\gamma_-}\int_0^\infty e^{-2\lambda t}\wt^{2l}
|\pd_t^{j_1}\pd_\tau^{j_2}\pd_n^{j_3}\tv(t)|^2,
\end{split}\end{equation}
which implies that
\begin{equation}\label{est_tv6}
\begin{split}
&\quad\|\tv\|_{\tilde\cB_\lam^{j,0}}^2+\lambda\|\tv\|_{\cB_\lam^{j,0}}^2+\|\tv\|_{\cB_\lam^{j,1}}^2\\
&\lesssim \|\tv\|_{\cB_\lam^{j,0}(0)}^2+
\|kf_1-f_2\|_{\cB_\lam^{j,0}}^2+\|\tv\|_{\cB_\lam^{3,1}}^2+\sum_{j_1+j_2+j_3\leq j}\int_{\pd Q_T^-} e^{-2\lambda t}\wt^{2l}
|\pd_t^{j_1}\pd_\tau^{j_2}\pd_n^{j_3}\tv|^2.
\end{split}\end{equation}
Since the functions $\pd_t^q\tv|_{t=0}$ and $~\pd_n^q\tv|_{\pd Q_T^-}~(q\geq0)$
can be represented by linear combinations of $(u_0^i,v_0^i)$ and $(u_1^i,v_1^i)~(0\leq i\leq q)$,
 from Assumption \ref{ass-linsta} we have
\[\|\tv\|_{\cB_\lam^{j,0}(0)}^2+\sum_{j_1+j_2+j_3\leq j}\int_{\pd Q_T^-} e^{-2\lambda t}\wt^{2l}
|\pd_t^{j_1}\pd_\tau^{j_2}\pd_n^{j_3}\tv|^2\lesssim \mathcal{M}_0^2.\]

Thus, from \eqref{est_tv6} we obtain
\begin{equation}\label{est_tv7}
\begin{split}
&\quad\|\tv\|_{\tilde\cB_\lam^{j,0}}^2+\lambda\|\tv\|_{\cB_\lam^{j,0}}^2+\|\tv\|_{\cB_\lam^{j,1}}^2\lesssim \mathcal{M}_0^2+
\|kf_1-f_2\|_{\cB_\lam^{j,0}}^2+\|\tv\|_{\cB_\lam^{3,1}}^2.
\end{split}\end{equation}

\underline{Step 3.} Estimates of normal derivatives.

From the equation $\eqref{eq_tv}_1$, we know that
\begin{equation}\label{def_vzz}
\pd_z^2\tv=\partial_t\tilde v+(u^s\partial_x+ku^s\partial_y+w^s\partial_z)\tilde
v+k_yu^s\tilde v-(kf_1-f_2),
\end{equation}
which implies
\[\begin{split}
\|\tv\|_{\cB_\lam^{j_1,2}}^2&\lesssim\|\tv\|_{\cB_\lam^{j_1+1,0}}^2+(\|u^s\|_{L^\infty}+\|ku^s\|_{L^\infty})\cdot\|\tv\|_{\cB_\lam^{j_1+1,0}}+\|w^s\|_{L^\infty}\cdot\|\tv\|_{\cB_\lam^{j_1,1}}\\
&\quad+(\|u^s\|_{\cD_0^{j_1}}+\|ku^s\|_{\cD_0^{j_1}}+\|w^s\|_{\cD_0^{j_1}})\cdot\|\tv\|_{\cB_\lam^{3,1}}
+\|k_yu^s\|_{L^\infty}\cdot\|\tv\|_{\cB_\lam^{j_1,0}}\\
&\quad+\|k_yu^s\|_{\cD_0^{j_1}}\cdot\|\tv\|_{\cB_\lam^{2,0}}+\|kf_1-f_2\|_{\cB_\lam^{j_1,0}}^2\\
&\lesssim \|\tv\|_{\cB_\lam^{j_1+1,0}}^2+\|\tv\|_{\cB_\lam^{j_1,1}}^2+\|\tv\|_{\cB_\lam^{3,1}}^2+\|kf_1-f_2\|_{\cB_\lam^{j_1,0}}^2~.
\end{split}\]
Combining the above inequality with \eqref{est_tv7}, it follows that
\begin{equation}\label{est_tv8}
\|\tv\|_{\cB_\lam^{j_1,2}}^2\lesssim \mathcal{M}_0^2+\|kf_1-f_2\|_{\cB_\lam^{j_1+1,0}}^2+\|kf_1-f_2\|_{\cB_\lam^{3,0}}^2~.
\end{equation}

For any fixed $j_2\ge 3$, applying the operator $\pd_\cT^\beta\pd_z^{j_2-2}~(|\beta|\leq j_1)$
to \eqref{def_vzz}, and using a similar argument as above, we get
\[
\|\tv\|_{\cB_\lam^{j_1,j_2}}^2\lesssim \|\tv\|_{\cB_\lam^{j_1+1,j_2-2}}^2+\|\tv\|_{\cB_\lam^{j_1,j_2-1}}^2+\|\tv\|_{\cB_\lam^{3,1}}^2+\|kf_1-f_2\|_{\cB_\lam^{j_1,j_2-2}}^2~,
\]
which implies that
\[
\|\tv\|_{\cB_{0,l}^{j_1,j_2}}^2\lesssim \|\tv\|_{\cB_{0,l}^{j_1+1,j_2-2}}^2+\|\tv\|_{\cB_{0,l}^{j_1,j_2-1}}^2+\|\tv\|_{\cB_{0,l}^{3,1}}^2+\|kf_1-f_2\|_{\cB_{0,l}^{j_1,j_2-2}}^2~.
\]
Therefore, we finally obtain that
\begin{equation}\label{est_tv9}
\|\tv\|_{\cA_l^j}\lesssim \mathcal{M}_0+\|kf_1-f_2\|_{\cA_l^j},
\end{equation}
which implies the estimates \eqref{est-linsta1} immediately.

\underline{Step 4.} Estimates of $\tilde v/\pd_z u^s$.

From the problem \eqref{eq_tv} of $\tilde v$, we know that $\tilde w\triangleq\tilde v/\pd_z u^s$ satisfies the following problem in $Q_T$:
\begin{equation}\label{pr_tw}
\begin{cases}
\pd_t \tilde w+(u^s\pd_x+k\pd_y u^s+w^s\pd_z)\tilde w-2\eta ~\pd_z \tilde w-\pd_z^2 \tilde w+(\zeta+k_y u^s)\tilde w=\frac{kf_1-f_2}{\pd_z u^s},\\
\tilde w|_{z=0}=0,\quad \lim\limits_{z\rightarrow+\infty}\tilde w=0,\quad \tilde w|_{\pd Q_T^-}=\frac{(ku_1-v_1)(t,x,y,z)}{\pd_z u^s(t,x,y,z)|_{\pd Q_T^-}},\\
\tilde w|_{t=0}=\frac{(ku_0-v_0)(x,y,z)}{\pd_z u^s(0,x,y,z)}.
\end{cases}\end{equation}
From Assumption \ref{ass-linsta} and by
a similar argument as given in  the above three steps for the problem \eqref{pr_tw} of $\tilde w$, one can obtain
\[
\|\tilde w\|_{\cA_l^{j-1}}\lesssim \mathcal{M}_1+\|{kf_1-f_2\over\pd_z u^s}\|_{\cA_l^{j-1}},
\]
from which the estimate \eqref{est-linsta1-1} follows. And this completes the proof of the lemma.
\end{proof}

Rewrite the problem \eqref{leqp}
by using that $v=ku-\tilde v$ as follows:
\begin{equation}\label{equw}
\begin{cases}
\partial_t u+(u^s\partial_x+ku^s\partial_y+w^s\partial_z)u+
  (u\partial_x+ku\partial_y+w\partial_z)u^s-\partial_z^2u=f_1+\tilde v\partial_yu^s,~in~Q_T,\\
\partial_x u+\partial_y (ku)+\partial_zw=\partial_y\tilde v,\quad in ~Q_T,\\
(u,w)|_{z=0}=0, \quad \lim\limits_{z\rightarrow +\infty}u=0,\quad
u|_{\pd Q_T^-}=u_1(t,x,y,z)|_{\pd Q_T^-},\\
u|_{t=0}=u_0(x,y,z).
\end{cases}
\end{equation}
As in \cite{AWXY}, for the problem \eqref{equw}, we introduce the transformation:
\begin{equation}\label{def_th}
h=\pd_z({u\over \pd_z u^s}),\qquad or\qquad u=\pd_z u^s\int_0^z hd\tilde z.
\end{equation}
Then, from \eqref{equw} we know that  $h(t,x,y,z)$ satisfies the following problem in $Q_T$:
\begin{equation}\label{eqh}
\begin{cases}
\partial_t h+[u^s\partial_x+k u^s
\partial_y+w^s\partial_z]h-2\partial_z(\eta
h)+\partial_z\big[(\zeta-k_y u^s)\int_0^z
  h ~ds\big]-\partial_z^2 h\\
  \qquad\qquad=\pd_z (\tilde f+\pd_y u^s\frac{\tilde v}{\pd_z u^s})-\pd_y \tilde v,\\[2mm]
  (\pd _z h+2\eta h)|_{z=0}=-\tilde f|_{z=0}, \quad h|_{\pd Q_T^-}=h_1(t,x,y,z)\triangleq  \pd_z({u_1(t,x,y,z)\over \pd_z u^s(t,x,y,z)|_{\pd Q_T^-}}),\\[2mm]
  h|_{t=0}=h_0(x,y,z)\triangleq \pd_z({u_0(x,y,z)\over\pd_z u^s(0,x,y,z)}),
\end{cases}
\end{equation}
where functions $\eta,\zeta,\tilde{f}$ are given in \eqref{def-linsta}.

Following the approach used in
 \cite{AWXY} and the proof of Lemma \ref{lem-linsta1}, we have the following
result on  the problem \eqref{eqh}.

\begin{lemma}\label{lem-linsta2}
Under Assumption \ref{ass-linsta}, the problem \eqref{eqh} has a unique
solution $h(t,x,y,z)$, and the following estimate holds:
\begin{equation}\label{pest-linsta2}
 \|h\|_{\cA_l^{j-1}}\leq C\Big(\mathcal{M}_1+\|\tilde f+\pd_y u^s\frac{\tilde v}{\pd_z u^s}\|_{\cA_l^{j-1}}
 +\|\pd_y \tilde v\|_{\cA_l^{j-1}}\Big)
\end{equation}
for a positive constant $C$.
\end{lemma}

Finally, by combining the results given in Lemmas \ref{lem-linsta1} and \ref{lem-linsta2}, we obtain classical solutions
$\tv\in\cA_l^j(Q_T)$ to the problem \eqref{eq_tv}, and $h\in\cA_l^{j-1}(Q_T)$ to the problem \eqref{eqh}, leading to
\[u=\pd_z u^s\int_0^z hd\tilde z,\qquad~v=k\pd_z u^s\int_0^z hd\tilde z-\tv,\]
and
\[w=-\int_0^z(\pd_x u+\pd_y v)d\tilde z,\]
from which we immediately obtain $(u, v)\in \cA_l^{j-1}(Q_T)$, $w\in\cD_0^{j-2}(Q_T)$, and the estimate \eqref{est_or}. It is straightforward to show that $(u,v,w)$ is the unique classical solution to the problem \eqref{leqp}. Thus, this concludes Theorem 3.2.


\section{Construction of approximate solutions to problem \eqref{eq_tr}}

Now, we will develop Oleinik's method \cite{Ole}
to construct the approximate solution sequence $\{W^n\}_{n\geq0}$ to the problem \eqref{eq_tr}.


\subsection{Construction of the zero-th order approximate solution}

In this subsection, we  construct the zero-th order approximate solution $W^0$ of the problem \eqref{eq_tr}.

To do this, we first introduce several notations for later use.\\

{\underline{\it Notations}:}
\begin{enumerate}
\item[(1)]
For the domain $D\subset\R^2$ with a smooth boundary $\partial D$,
set
\[\gamma_+~=~\{(\xi,\eta)\in\pd D: (1,k(\xi,\eta))\cdot\vec{n}(\xi,\eta)>0\}.\]
\item[(2)]
For a sufficiently small number $\rho>0$,
denote by
\[\Gamma_\rho=\{(\xi,\eta)\in \R^2\setminus D:~|(\xi,\eta)-\pd D|<\rho\},\]
\[\Gamma^-_\rho=\{(\xi,\eta)\in \R^2\setminus D:~|(\xi,\eta)-\gamma_-|<\rho\}\subset \Gamma_\rho,\]
with $d(\xi,\eta)$ denoting the distance from $(\xi,\eta)$ to $\pd D$ and $P(\xi,\eta)$ the
point of $\pd D$ closest to $(\xi,\eta)$.

\item[(3)] In the $(\xi,\eta)-$plane,
let $\widetilde{D}$ be an infinitely differentiable bounded domain satisfying
\[D\cup \Gamma^-_{\sigma/2}~\subset~\widetilde{D}~\subset~D\cup \Gamma_\sigma,\]
for a fixed $0<\sigma<{\rho\over2}$, and $D^*$ is a simply connected domain with $C^1$ boundary satisfying
\[D\cup\gamma_+~\subset~ D^*~\subset~\widetilde{D},\quad \gamma_-~\subset~\pd D\cap\pd D^*.\]
Furthermore,
 there is a smooth extension $(k_1(\xi,\eta), k_2(\xi,\eta))$ of the vector field $(1,k(\xi,\eta))$ from the domain $D$ to the domain $D^*$, such that
\begin{equation}\label{ass_k12}
(k_1(\xi, \eta),k_2(\xi,\eta))\cdot\vec{n}(\xi,\eta)~\leq~0,\quad \forall (\xi,\eta)\in \pd D^*,\end{equation}
where $\vec{n}(\xi,\eta)$ is the outward normal vector on $ \pd D^*$.

\item[(4)]  In the $(\xi, \eta, \zeta)-$space, let $G$ be a
simply
connected smooth domain, satisfying
\[\widetilde{D}\times[0,1+\delta]_\zeta~\subset~G~\subset~(D\cup \Gamma_\rho)\times[0,1+\delta]_\zeta,\]
for a small fixed number $\delta>0$.

\item[(5)] In the $(t, \xi, \eta, \zeta)-$space,
denote by $\Sigma$ a smooth bounded domain, satisfying
\[[-1,T+1]_t\times G~\subset~\Sigma~\subset~[-2,T+2]_t\times G,\]
and
\[\Sigma^*:=\Big(\Sigma\cap\{0\leq t< T,(\xi,\eta)\in D^*\}\Big)~\bigcup~\Big(\Sigma\cap\{ t\geq T\}\Big).\]

\end{enumerate}

\begin{remark}

(1) From \cite{Fife}, we know that $d(\xi,\eta)$ and $P(\xi,\eta)$ are uniquely defined for $(\xi,\eta)\in \Gamma_\rho$ if $\rho$ is properly small.\\
(2) From the condition $(H1)$ given in the introduction, we know that the domain $D^*$ and functions
$k_1(\xi,\eta),k_2(\xi,\eta)$ are well defined.\\

(3) From the above notations, obviously
we have
\[\pd \Sigma\cap\{-1\leq t\leq T+1\}~=~[-1,T+1]\times\pd G,\]
and $\Omega~\subset~\Sigma^*$,
\(\Gamma_-~\subset~\pd \Omega\cap\pd \Sigma^*.\)
\end{remark}

Set
\[\begin{split}
&S_0:=\{t=0,~(\xi,\eta)\in D,~0\leq\zeta\leq 1\},\\
&S_1:=\{0\leq t \leq T,~(\xi,\eta)\in \gamma_-,~0\leq\zeta\leq 1\},
\end{split}
\]
and $N_\delta(S_0)$ ($N_\delta(S_1)$ resp.) a $\delta-$neighborhood of $S_0$ ($S_1$ resp.) in $(t,\xi,\eta,\zeta)$-space.

To construct $W^0$, we first define $W^*$ as
\begin{equation}\label{w*}
W^*(t,\xi,\eta,\zeta)=W_0(\xi,\eta,\zeta)+t\cdot W_0^1(\xi,\eta,\zeta)+{t^2\over 2!}\cdot W_0^2(\xi,\eta,\zeta)+\cdots+{t^4\over 4!}\cdot W_0^4(\xi,\eta,\zeta),
\end{equation}
in $N_\delta(S_0)\cap\{t\leq0, (\xi,\eta)\in D, 0\le \zeta\le 1\}$, and
\begin{equation}\label{w*1}
\begin{split}
W^*(t,\xi,\eta,\zeta)=& W_1(t,P(\xi,\eta),\zeta)+d(\xi,\eta)\cdot W_1^1(t,P(\xi,\eta),\zeta)\\
&+{d(\xi,\eta)^2\over 2!}\cdot W_1^2(t,P(\xi,\eta),\zeta)+\cdots+{d(\xi,\eta)^4\over 4!}\cdot W_1^4(t,P(\xi,\eta),\zeta),
\end{split}\end{equation}
in $N_\delta(S_1)\cap\{t\geq0, (\xi,\eta)\in \widetilde{D}\cap \Gamma^-_\rho, 0\le \zeta\le 1\}$,  where
$W_0^i$, $W_1^j$ ($1\le i, j\le 4$) are given in \eqref{def_w0} and \eqref{def_w1} respectively.

Now, we extend the function $W^*$, given in \eqref{w*} and \eqref{w*1} near the boundary $S_0$ and $S_1$, smoothly into the remaining part of the region $\Sigma\setminus \Sigma^*$,
such that $W^*\in C^6(\overline{\Sigma\setminus \Sigma^*})$ and $W^*$ is infinitely differentiable away from the boundary $S_0\cup S_1$.
Such function $W^*$ can be constructed by using Assumption \ref{ass_1}, and it follows immediately that
\[\pd_t^i W^*~=~W_0^i,\quad on~S_0;\qquad \pd_n^j W^*~=~W_1^j,\quad on~S_1,\]
for all $0\leq i,j\leq 4$.

We extend $W^*$ smoothly into $\Sigma^*$, which is still denoted by $W^*$ for simplicity, such that $W^*$ has bounded derivatives up to order four in $\Sigma$, and
\[M^{-1}(1-\zeta)\leq W^*\leq M(1-\zeta), \quad in~\Omega,\]
for the positive constant $M$ given in \eqref{seq_ass}.

Finally, by letting $W^0=W^*|_{\Omega}$, we get that $W^0(t,\xi,\eta,\zeta)$ satisfies the conditions given in \eqref{w0}.


\subsection{Construction of the $n-$th order approximate solution}

In this subsection, we will construct the approximate solution $W^n$ to the linearized problem \eqref{eq_seq}.
Precisely, a sequence of functions $\{W^n(t,\xi,\eta,\zeta)\}_{n\geq0}$
will be constructed by induction on $n$,
 in the region $\Sigma$, satisfying
the following properties:
\begin{enumerate}
\item when $n=0$, $W^0=W^*$ for the function $W^*$ constructed in  Section 2.2;
\item
$W^n(t,\xi,\eta,\zeta)~(n\geq0)$ has continuous bounded derivatives in $\Sigma$ up to order three,
and the third order derivatives are Lipschitz continuous;
\item for all $n\ge 0$, \[\begin{split}
&W^n(t,\xi,\eta,\zeta)=W^*(t,\xi,\eta,\zeta),\quad in ~\Sigma\setminus \Sigma^*,\\
&W^n(t,\xi,\eta,1)~=~0,\quad \forall~t\in(0,T),~ (\xi,\eta)\in D;
\end{split}\]
\item for all $n\ge 1$, the functions $W^n(t,\xi,\eta,\zeta)$ satisfy the problem \eqref{eq_seq} in $\Omega$.
\end{enumerate}

For any fixed $n\geq1$, suppose that $W^{n-1}$
 satisfies the above four properties,
we will  verify that $W^n$ satisfies the same properties. Note that
the coefficient of the zero-th order term in the first equation of $\eqref{eq_seq}$
may vanish. Set
\begin{equation}\label{tw}
\W^n(t,\xi,\eta,\zeta)=e^{-\lambda t}W^n(t,\xi,\eta,\zeta)
\end{equation}
with a constant $\lambda>\|B\|_{L^\infty(\Omega)}$, then from \eqref{eq_seq} we know that $\W^n(t,\xi,\eta,\zeta)$ satisfies the following problem in $\Omega,$
\begin{equation}\label{eq_seq1}
\begin{cases}
\partial_t \widetilde{W}^n+\zeta U(\partial_\xi +k\partial_\eta)\widetilde{W}^n+A\pd_\zeta \widetilde{W}^n+(B+\lambda)\widetilde{W}^n-(W^{n-1})^2\pd_\zeta^2 \widetilde{W}^n=0,\\
 W^{n-1}\pd_\zeta \widetilde{W}^n|_{\zeta=0}=e^{-\lambda t}{ p_x\over U},
~\widetilde{W}^n|_{\Gamma_-}=e^{-\lambda t}W_1(t,\xi,\eta,\zeta),\\
\widetilde{W}^n|_{t=0}= W_0(\xi,\eta,\zeta).
\end{cases}\end{equation}

As in \cite{OA1},  introduce an elliptic operator in $\Sigma,$
\begin{equation}\label{op_seq}
\begin{split}
L^\epsilon(w)~\triangleq~&-\epsilon\triangle w-a_1\pd_t^2 w-a_2\pd_\xi^2 w-a_3\pd_\eta^2 w-[a_4+(W^{n-1})_\epsilon^2]\pd_\zeta^2 w\\
&+\pd_t w+\zeta U_\epsilon
(k_{1,\epsilon}\pd_\xi+k_{2,\epsilon}\pd_\eta)
w+A_\epsilon\pd_\zeta w+[B'_\epsilon+2(a_1+\epsilon)] w,
\end{split}\end{equation}
for a small parameter $\epsilon>0$, where
\[\triangle w~\triangleq~\pd_t^2 w+\pd_\xi^2 w+\pd_\eta^2 w+\pd_\zeta^2 w.\]
Here,
the notation $f_\epsilon$ denotes a regularization of the function $f$ by means of convolution with a non-negative $C^\infty$ function compactly supported in a ball of radius $\epsilon$. Moreover,
the functions $U$, $A$ and $B'$ are smooth extension of the corresponding functions and $\lambda+B$ from $\Omega$ to $\Sigma$ such that they are in $C^5(\Sigma)$ and $B'>0$; functions $k_1(\xi,\eta)$ and $k_2(\xi,\eta)$ are given in Notations (3) in the Subsection 2.2; the non-negative functions $a_i\in C^\infty(\Sigma),~1\le i\le 4$ satisfy the following conditions:
\begin{enumerate}
\item
$a_i>0~ (1\le i\le 4)$ for $ t<-1/2$  and $t>T+\frac{1}{2}$;
\item $a_2,a_3>0$ in the $\delta$-neighborhood of the boundary $[-1,T+1]\times(\pd G\setminus\{\zeta=0,or,1+\delta\})$;
\item  $a_4>0$  in the $\delta$-neighborhood of the boundary $[-1,T+1]\times(\pd G\setminus\{\zeta=0\})$;
\item $a_i=0$  on the rest of $\Sigma$  for  all $1\le i\le 4$.
\end{enumerate}

Then, consider the following elliptic problem with the Neumann boundary condition:
\begin{equation}\label{pr_n}
\begin{split}
&L^\epsilon(\W_\epsilon^n)~=~F_\epsilon, \quad in\quad \Sigma,\\
&{\pd \W_\epsilon^n\over\pd n}~=~f_\epsilon,\qquad on\quad \pd \Sigma,
\end{split}\end{equation}
where $\vec{n}$ is the unit outward normal
vector on $\pd \Sigma$,
 and the functions $F$ and $f$ are defined by
\[
F~=~\begin{cases}0, \qquad\qquad\qquad\qquad\qquad\quad\,~{\rm in}~\Omega;\\
L^\epsilon(e^{-\lambda t}W^*)+\epsilon\Delta (e^{-\lambda t}W^*)+[(W^*)^2-(W^{n-1})_\epsilon^2]\partial_\zeta^2 (e^{-\lambda t}W^*)
\\
\qquad\qquad+
[\zeta U(k_1\pd_\xi+k_2\pd_\eta)-\zeta U_\epsilon(k_{1,\epsilon}\pd_\xi+k_{2,\epsilon}\pd_\eta)](e^{-\lambda t}W^*)
\\
\qquad\qquad
+
[(A-A_\epsilon)\pd_\zeta +B'-B'_\epsilon-2\epsilon](e^{-\lambda t}W^*),\qquad{\rm in}~
\Sigma\setminus \Sigma^*;
\\
{\rm smooth ~connection},\qquad\quad\,\ {\rm in~the~rest~of}~\Sigma,
\end{cases}
\]
and
\[
f~=~\begin{cases} -e^{-\lambda t}{p_x\over UW^{n-1}},~&{\rm on}~S_3;\\
\frac{\pd}{\pd n} (e^{-\lambda t}W^*),
~ &{\rm on}~\pd Q\cap\pd(\Sigma\setminus \Sigma^*);\\
{\rm smooth ~connection},~ &{\rm on~the~rest~of}~\pd \Sigma,
\end{cases}
\]
where
\[S_3:=\{0\leq t\leq T,~(\xi,\eta)\in D,~\zeta=0\}.\]
Moreover, from the construction of $W^*$ defined in
$\Sigma\setminus \Sigma^*$,
 we can assume that the function $F$ has bounded derivatives up to
order four
in $\Sigma$ and is infinitely differentiable outside a $\delta$-neighborhood of $\Omega$; the function $f$ also has bounded derivatives up to
order four  in a neighborhood of $S_3$ and is infinitely differentiable on the rest of $\pd \Sigma$.

The boundary value problem \eqref{pr_n} has a unique solution $\W^n_\epsilon$ in the region $\Sigma$
by using the classical theory of elliptic equations, cf.  \cite[Theorem 3.6]{ADN} and  \cite[Theorem 12.7]{GT},
by noting that the coefficients and the right hand sides of the problem \eqref{pr_n} are smooth in $\Sigma$ and the coefficient of the zero-th order of $\W^n_\epsilon$ is positive,

Then, we show that
the derivatives up to order four of function $\W^n_\epsilon$ are uniformly
bounded in $\epsilon$. We establish the following proposition for $\W^n_\epsilon$.

\begin{prop}\label{prop_n2}
In the domain $\Sigma$, the solution $\W^n_\epsilon$ of the problem \eqref{pr_n} and its derivatives, up to order four, are bounded uniformly in $\epsilon$.
\end{prop}

At this moment, we first assume that
Proposition \ref{prop_n2} is true, which will be studied later. And we are going to prove the following proposition
 from which the existence of the solution $W^n$ to the problem \eqref{eq_seq} follows immediately.

\begin{prop}\label{lem_n3}
There exists a function $W^n(t,\xi,\eta,\zeta)$ in $\Sigma$ such that $W^n$ has continuous derivatives up to order three, and the third order derivatives of $W^n$
are Lipschitz continuous. Moreover, the restriction of $W^n$ in $\Omega$ is a solution to the problem \eqref{eq_seq}, and
$W^n=W^*$ in $\Sigma\setminus \Sigma^*$,
\begin{equation}\label{bd_n3}
W^n(t,\xi,\eta,1)~=~0,\quad \forall~t\in(0,T),~(\xi,\eta)\in D.
\end{equation}
\end{prop}

\begin{proof}[{\bf{Proof.}}]
It is sufficient to prove the existence of the new unknown function $\W^n=e^{-\lambda t}W^n$ which satisfies the corresponding properties,
and this will be done in the following several steps.

\underline{Step 1.}
From the above hypothesis that the derivatives of the solution $\W^n_\epsilon$ to \eqref{pr_n} up to
order four  are bounded uniformly in $\epsilon$,
there exists a subsequence $\{\W^n_{\epsilon_k}(t,\xi,\eta,\zeta)\}_{k>0}$
such that $\W^n_{\epsilon_k}(t,\xi,\eta,\zeta)$
 converges in $C^3(\Sigma)$ to $\W^n(t,\xi,\eta,\zeta)$ uniformly
 in $\Sigma$ as $\epsilon_k\rightarrow0$, and the third order derivatives of
 $\W^n(t,\xi,\eta,\zeta)$ are Lipschitz continuous.
From the special form of the problem \eqref{pr_n},
it is easy to see
 that $\W^n(t,\xi,\eta,\zeta)$ satisfies the equation and the boundary condition at $\{\zeta=0\}$ given
in the problem \eqref{eq_seq1}.
It remains to verify that $\W^n$ satisfies
the other boundary conditions given in \eqref{eq_seq1}, and $\W^n=e^{-\lambda t}W^*$ in $\Sigma\setminus \Sigma^*$,
\begin{equation*}
\W^n(t,\xi,\eta,1)~=~0,\quad \forall~t\in(0,T),~(\xi,\eta)\in D.
\end{equation*}

\underline{Step 2.}
In this step, we  prove that
$\W^n=e^{-\lambda t}W^*$ in $\Sigma\setminus \Sigma^*$,
from which the boundary condition on $\Gamma_-$ and the initial data on $\{t=0\}$ given in \eqref{eq_seq1} follow immediately.

Setting $V=\W^n-e^{-\lambda t}W^*$,
then from \eqref{pr_n}, $V$ satisfies
\begin{equation}\label{eq-V}
\begin{cases}-a_1\pd_t^2 V-a_2\pd_\xi^2 V-a_3\pd_\eta^2 V-\Big[a_4+(W^*)^2\Big]\pd_\zeta^2 V+V_t\\
\hspace{.3in}+
\zeta U (k_1\cdot V_\xi+k_2\cdot V_\eta )
+A V_\zeta+(B'+2a_1)V=0,\qquad {\rm in} ~\Sigma\setminus \Sigma^*,\\
{\pd V\over\pd n}=0,\quad {\rm on} ~\pd(\Sigma\setminus \Sigma^*)\cap\pd \Sigma.
\end{cases}\end{equation}

Let $E(t,\xi,\eta,\zeta)$ be a smooth function in $\Sigma$ such that $\pd E/\pd n<0$ on $\pd \Sigma$ and $E>1$. Set
\[V_1~\triangleq~V(E+C),\]
for a positive constant $C>0$.
 It is easy to check that $V_1$ satisfies an equation similar to that of $V$ given in \eqref{eq-V}, and the zero-th order coefficient of $V_1$ is positive if $C$ is sufficiently large. The boundary condition on $\pd(\Sigma\setminus \Sigma^*)\cap\pd \Sigma$
 for $V_1$ becomes
\begin{equation}\label{bd_n4}
\frac{\pd V_1}{\pd n}+\alpha_1 V_1=0,\end{equation}
with $\alpha_1=-{1\over E+C}\cdot{\pd E\over\pd n}>0$. Thus, $|V_1|$ does not achieve its non-zero maximum on the boundary $\pd(\Sigma\setminus \Sigma^*)\cap\pd \Sigma$.
Otherwise, at the point of maximum of $|V_1|$ on the boundary $\pd(\Sigma\setminus \Sigma^*)\cap\pd \Sigma$, we must have
\[V_1\frac{\pd V_1}{\pd n}+\alpha_1 (V_1)^2>0,
\]
which is a contradiction to \eqref{bd_n4}.

Similarly, the non-zero maximum of $|V_1|$ is not attained in the interior of $\Sigma\setminus \Sigma^*$ nor on the boundary
$\pd(\Sigma\setminus \Sigma^*)\cap \left(\{t=0\}\cup
 \{(\xi,\eta)\in\pd D^*\}\right)$. Indeed,
if $|V_1|$ attains the maximal at $(t,\xi,\eta,\zeta)\in \pd(\Sigma\setminus \Sigma^*)\cap
\{(\xi,\eta)\in\pd D^*\}$, then at this point
\begin{equation}\label{eq-V2-1}
V_1\pd_t V_1\geq0,~V_1\pd_\tau V_1=0,~V_1\pd_n V_1\leq0,~V_1\pd_\zeta V_2=0,\end{equation}
which implies that
\begin{equation}\label{eq-V2-2}
\begin{split}
V_1(k_1\cdot\pd_\xi V_1+k_2\cdot\pd_\eta V_1)&=V_1\cdot\Big[\big((k_1,k_2)\cdot\vec{\tau}\big)\pd_\tau V_1+\big((k_1,k_2)\cdot\vec{n}\big)\pd_n V_1\Big]\\
&=\Big((k_1,k_2)\cdot\vec{n}\Big)V_1\pd_n V_1\geq0
\end{split}\end{equation}
by using \eqref{ass_k12}.
For the second order derivatives, we have
\[V_1\pd_t^2V_1\leq0,\quad V_1\pd_\zeta^2 V_1\leq0.\]
Noting that $a_2=a_3 =0
$ at such maximal point, and $V_1$ satisfies an equation similar to \eqref{eq-V}
with the zero-th order coefficient being positive. Hence, there is  a contradiction. Therefore,
we have $V_1\equiv0$ and then $\W^n\equiv e^{-\lambda t}W^*$ in $\Sigma\setminus \Sigma^*$, which implies
that $\W^n$ satisfies the boundary conditions on $\{t=0\}$ and $\Gamma_-$ given in the problem \eqref{eq_seq1}.

\underline{Step 3.}
It remains to show that
\begin{equation}\label{topbd}
\W^n(t,\xi,\eta,1)~=~0,\quad {\rm for~all}~(t,\xi,\eta)\in(0,T)\times D.\end{equation}
From the first step, we know that $\W^n$ is a classical solution to the problem \eqref{eq_seq1}.
Restricting
the problem \eqref{eq_seq1} on the plane $\{\zeta=1\}$, it follows that $\w^n~\triangleq~\W^n(t,\xi,\eta,1)$ satisfies the following problem in $\{(t,\xi,\eta):~t\in(0,T),~(\xi,\eta)\in D\},$
\begin{equation}\label{eq_tw}
\begin{cases}
\partial_t \w^n+U(\partial_\xi +k\partial_\eta)\w^n+b\w^n=0,\\
\w^n|_{t=0}=W_0(\xi,\eta,1)=0,\quad \w^n|_{\Gamma_-}=e^{-\lambda t}W_1(t,\xi,\eta,1)=0.
\end{cases}
\end{equation}
by using $A|_{\zeta=1}=0$ and the induction assumption $W^{n-1}|_{\zeta=1}=0$. Here,
$b=\lambda+B(t,\xi,\eta,1)>0$.
It follows that $\w^n\equiv0$, which implies that \eqref{topbd} holds.

\end{proof}

\begin{remark}
The identity \eqref{bd_n3} explains why no condition
 on the boundary $\{\zeta=1\}$ of $\Omega$ in the problem \eqref{eq_seq} is needed.
\end{remark}

We now come back to give the proof of Proposition \ref{prop_n2}, which contains the following three lemmas.

\begin{lemma}\label{lem_n1}
There exists a positive constant $M_0$, independent of $\epsilon$, such that the solution $\W^n_\epsilon$
to the problem \eqref{pr_n} satisfies:
\begin{equation}\label{bd_n1}
\big|\W^n_\epsilon(t,\xi,\eta,\zeta)\big|~\leq~M_0,\,\quad {\rm in}~\Sigma.
\end{equation}
\end{lemma}

This lemma can be obtained by applying the maximal principle for the problem \eqref{pr_n} and using the properties of $W^*$ given in Section 2.2, a similar result was given in \cite[Lemma 4.3.9]{Ole}, so we omit the proof.

\begin{lemma}\label{lem_n2-1} For a given positive constant $r_1<\frac{1}{2}$, in the domain
$\Sigma_{out}=\Sigma\cap\{t<-\frac{1}{2}-r_1,~or ~t>T+\frac{1}{2}+r_1\}$,
the solution $\W^n_\epsilon$ to the problem \eqref{pr_n} has bounded derivatives up to
order four uniformly in $\epsilon$.
\end{lemma}

\begin{proof}[\bf{Proof.}]
Noting that in $\Sigma_{out}$, the equation in \eqref{pr_n} is uniformly elliptic with respect to $\epsilon$, by applying  the well-known
Schauder type estimates,
cf. \cite{ADN} and \cite{GT},  in $\Sigma_{out}$, the derivatives of $\W^n_\epsilon$ up to order four are bounded uniformly  in $\epsilon$, by using the induction hypothesis that $W^{n-1}$ has bounded derivatives up to order three.
\end{proof}

To conclude Proposition \ref{prop_n2}, it suffices to show

\begin{lemma}\label{lem_n2-3}  In the domain $\Sigma_{int}=\Sigma\cap\{-{1\over2}-r_1\leq t\leq T+\frac{1}{2}+r_1\}$, the
solution $\W^n_\epsilon$ of the problem \eqref{pr_n} has bounded derivatives up to
order four  uniformly in $\epsilon$.
\end{lemma}

Based on the above three lemmas, 
Proposition \ref{prop_n2} follows immediately,  our remaining task is to prove Lemma \ref{lem_n2-3}. For this,
let $\pd'G=\pd G\setminus\{\zeta=0,~or ~1+\delta\}$, and we first give an estimate of $\W^n_\epsilon$ on  part of
the boundary $\pd \Sigma_{int}$ in the following lemma.

\begin{lemma}\label{lem_n2-2}
The derivatives of the solution $\W^n_\epsilon$ to the problem \eqref{pr_n}, up to order four,  are
 bounded uniformly in $\epsilon$ on the boundary $[-{1\over2}-r_1,T+\frac{1}{2}+r_1]\times\pd'G.$
\end{lemma}

We first  conclude Lemma \ref{lem_n2-3} by assuming that
the assertion of Lemma \ref{lem_n2-2} is true.

\begin{proof}[\bf{Proof of Lemma \ref{lem_n2-3}.}]
As in \cite{Ole}, by setting
\[V=\W^n_\epsilon e^{\psi(\zeta)},\quad \psi(\zeta)={\alpha\zeta(1+\delta-\zeta)\over1+\delta}\]
with a positive constant $\alpha>0$,
we get that from the problem \eqref{pr_n} of $\W^n_\epsilon$, $V$ satisfies the following boundary conditions:
\begin{equation}\label{bd_v}\begin{cases}
\pd_\zeta V-\alpha V~=~-f_\epsilon,\qquad &{\rm on}\quad\{\zeta=0\},\\
\pd_\zeta V+\alpha V~=~f_\epsilon,\qquad &{\rm on} \quad\{\zeta=1+\delta\}.
\end{cases}\end{equation}

To estimate the first order derivatives of $V$ in $\Sigma_{int}$, define
\begin{equation}\label{def_Pi}
\Pi_1~=~V_t^2+V_\xi^2+V_\eta^2+V_\zeta(V_\zeta-2Y)+k(\zeta),
\end{equation}
where
\begin{equation}\label{Y}
Y=(\alpha V-f_\epsilon)\varphi(\zeta),
\end{equation}
with a smooth function $\varphi$ satisfying
\[\varphi(\zeta)=\begin{cases}1,\quad &for ~|\zeta|\leq\delta/4,\\
-1,\quad &for~|1+\delta-\zeta|\leq\delta/4,\\
0,\quad &for ~{\delta\over2}\leq\zeta\leq1+{\delta\over2},
\end{cases}\]
and $k(\zeta)$ is a positive function to be chosen later. Then, from the boundary conditions \eqref{bd_v} and the definition \eqref{Y} of $Y$, we have
\[\pd_\zeta V~=~Y,\quad for~\quad\zeta=0,~or~1+\delta.\]
Hence, we have that on $\{\zeta=0\}$,
\[\begin{split}\pd_\zeta\Pi_1&=2V_t V_{t\zeta}+2V_\xi V_{\xi\zeta}+2V_\eta V_{\eta\zeta}-2V_\zeta Y_\zeta+k'(0)\\
&=2\alpha\Big(V_t^2+V_\xi^2+V_\eta^2\Big)-2YY_\zeta-2V_t(f_\epsilon)_t
-2V_\xi(f_\epsilon)_\xi-2V_\eta(f_\epsilon)_\eta+k'(0),
\end{split}\]
which implies that by requiring $k'(0)>0$ large enough,
\begin{equation}\label{zeta-0}
\pd_\zeta \Pi_1\Big|_{\zeta=0}>0.
\end{equation}
 Similarly, by choosing $k'(1+\delta)<0$ and its absolute value being sufficiently large, we have
\begin{equation}\label{zeta-1}
\pd_\zeta \Pi_1\Big|_{\zeta=1+\delta}<0.
\end{equation}
Therefore, the maximum of $\Pi_1$ can not be attained on the boundary $\{\zeta=0\}\cup\{\zeta=1+\delta\}$.

By direct calculation, there exist positive constants $C_1$ and $C_2$, independent of $\epsilon$, such that
\begin{equation}\label{eq_v1}
\tilde{L}(\Pi_1)+C_1\Pi_1\leq C_2
\end{equation}
with
\[\begin{split}\tilde{L}(w)~\triangleq~&L^\epsilon(w)-2\Big[(W^{n-1})_\epsilon^2+(a_4+\epsilon)\Big]
\psi_\zeta \cdot\pd_\zeta w\\
&+\Big\{A_\epsilon\psi_\zeta-\big[(W^{n-1})_\epsilon^2+(a_4+\epsilon)\big]
\cdot(\psi_{\zeta\zeta}+\psi_\zeta^2)\Big\}\cdot w.
\end{split}\]

Next, in $\Sigma_{int}$, by setting
\[\Pi_1^*~=~\Pi_1e^{-\beta t}\]
for a constant $\beta>0$,
it is easy to deduce that $\Pi_1^*$ satisfies a differential inequality similar to that one given in \eqref{eq_v1}, in which the zero-th order coefficient of $\Pi_1^*$ is larger than one for sufficiently small $\epsilon$ when $\beta$ is suitably large and $a_1$ is chosen suitably small. Therefore, from this differential inequality we obtain that
if $\Pi_1^*$ attains its maximum in the interior of
 $\Sigma_{int}$, then $\Pi_1^*$ is bounded by a constant independent of $\epsilon$.

Next, from \eqref{zeta-0} and \eqref{zeta-1}
we know that $\Pi_1^*$ does not attain its maximum on the boundary $\{\zeta=0\}\cup\{\zeta=1+\delta\}$.
On the other hand, by using Lemmas \ref{lem_n2-1} and \ref{lem_n2-3}, we obtain that $\Pi_1^*$ is uniformly bounded in $\epsilon$ on the other boundaries of $\Sigma_{int}$.
In summary,  we conclude that $\Pi_1^*$ is uniformly bounded in $\epsilon$ in $\Sigma_{int}$, so is $\Pi_1$.
Thus, it follows that the first order derivatives of $\W^n_\epsilon$ are uniformly bounded in $\epsilon$ in $\Sigma_{int}$.

Similarly, we can estimate the second and the third order derivatives of $V$ in $\Sigma_{int}$ by considering the following functionals:
\[\begin{split}\Pi_2~=~&\sum_{|\gamma|=2}(\pd_\cT^\gamma V)^2+\sum_{|\gamma|=1}\pd_\cT^\gamma V_{\zeta}\cdot\Big(\pd_\cT^\gamma V_{\zeta}-2\pd_\cT^\gamma Y\Big)
+g^2(\zeta)\cdot V_{\zeta\zeta}^2+k(\zeta),\end{split}\]
and
\[\begin{split}\Pi_3~=~&\sum_{|\gamma|=3}(\pd_\cT^\gamma V)^2+\sum_{|\gamma|=2}\pd_\cT^\gamma V_{\zeta}\cdot\Big(\pd_\cT^\gamma V_{\zeta}-2\pd_\cT^\gamma Y\Big)+g^2(\zeta)\cdot\sum_{|\gamma|=1}(\pd_\cT^\gamma V_{\zeta\zeta})^2+k(\zeta),
\end{split}\]
where
\begin{equation}\label{def_tanop-1}
\pd_\cT^\gamma=\pd_t^{\gamma_1}\pd_\xi^{\gamma_2}\pd_\eta^{\gamma_3},\quad \gamma=(\gamma_1,\gamma_2,\gamma_3),\quad |\gamma|=\gamma_1+\gamma_2+\gamma_3,
\end{equation}
denotes the differential operator tangential to the boundaries $\{\zeta=0\}\cup\{\zeta=1\}$,
and
$g(\zeta)$ is a smooth function satisfying:
\[g(\zeta)~=~\begin{cases}0,\quad &for~0\leq\zeta\leq{\delta\over4},~or~1+{3\delta\over4}\leq\zeta\leq1+\delta,\\
1,\quad &for~{\delta\over2}\leq\zeta\leq1+{\delta\over2}.
\end{cases}\]
The boundedness estimates on $\Pi_2$ and $\Pi_3$ can be derived in a way similar
 to the above discussion
 for $\Pi_1$. For this, one can deduce differential inequalities of $\Pi_2$ and $\Pi_3$ similar to the one given in  \eqref{eq_v1} for $\Pi_1$,
 by using the fact that the coefficient of $\pd_\zeta^2 w$ in \eqref{op_seq} is negative when $\zeta<\frac{\delta}{2}$ or $\zeta>1+\frac{\delta}{2}$.
 Thus, we can show that $\Pi_2$ and $\Pi_3$ are uniformly bounded in $\epsilon$ in $\Sigma_{int}$. Then,
the boundedness of $\Pi_2$ implies that the second order derivatives of $\W^n_\epsilon$, except $\pd_\zeta^2 \W^n_\epsilon$, are uniformly bounded in $\epsilon$ in $\Sigma_{int}$,
 and $\pd_\zeta^2 \W^n_\epsilon$ are uniformly bounded in $\epsilon$ in $\Sigma_{int}\cap\{{\delta\over2}\leq\zeta\leq1+{\delta\over2}\}$.

On the other hand, by using that the coefficient of $\pd_\zeta^2 w$ in \eqref{op_seq} is negative for $\zeta<\frac{\delta}{2}$ or $\zeta>1+\frac{\delta}{2}$ again,
 it follows that $\pd_\zeta^2 \W^n_\epsilon$ is uniformly bounded for $\zeta<{\delta\over2}$ or $\zeta>1+{\delta\over2}$.
 In summary, we have deduced that all second order derivatives of $\W^n_\epsilon$ are uniformly bounded in $\epsilon$ in $\Sigma_{int}$.
 Likewise, we can obtain the uniform boundedness of the third order derivatives of $\W^n_\epsilon$ by using a similar argument for $\Pi_3$.

In order to study the fourth order derivatives of $V$, set
\[\begin{split}
\Pi_4~=~&\sum_{|\gamma|=4}(\pd_\cT^\gamma V)^2+\sum_{|\gamma|=3}\pd_\cT^\gamma V_{\zeta}\cdot\Big(\pd_\cT^\gamma V_{\zeta}-2\pd_\cT^\gamma Y\Big)\\
&
+g^2(\zeta)\cdot\sum_{|\gamma|+i=4,~i\geq2}(\pd_\cT^\gamma\pd_\zeta^i V)^2
+k(\zeta).
\end{split}\]
To derive a differential inequality for $\Pi_4$  similar to \eqref{eq_v1}, we need to estimate
the terms $\tilde{L}(\pd^\gamma Y)$ with $|\gamma|=3$, which contain the fifth order derivatives of $f_\epsilon$ from the definition of $Y$ given in \eqref{Y}.
Since $f$ has bounded derivatives up to order four in a neighborhood of $S_3$ and is infinitely differentiable on the rest of $\pd \Sigma$,
the derivatives of $f_\epsilon$ up to order four
are bounded uniformly in $\epsilon$ in $\Sigma$, but the fifth order derivatives of $f_\epsilon$ has the order of
$\mathcal{O}(\epsilon^{-1})$ in the neighborhood of $S_3$. Note that in the neighborhood of $S_3$,
the second order derivatives in the operator $\tilde{L}$ have
 the coefficient $\epsilon$, that is,
\[\epsilon\pd_t^2,\quad\epsilon\pd_\xi^2,\quad\epsilon\pd_\eta^2.\]
Therefore, it is  uniformly bounded in $\epsilon$
 when applying the operator $\tilde{L}$ to the third order derivatives of $f_\epsilon$.
By studying $\Pi_4$ in a way similar to that given for $\Pi_1$,  we deduce that  $\Pi_4$ is bounded uniformly
 in $\epsilon$ in $\Sigma_{int}$, which implies that the fourth order derivatives of $\W^n_\epsilon$ are uniformly bounded in $\epsilon$ in $\Sigma_{int}$. Thus, we complete the proof of Lemma \ref{lem_n2-3}.
\end{proof}

We now turn to prove Lemma \ref{lem_n2-2}.

\begin{proof}[\bf{Proof of Lemma \ref{lem_n2-2}.}]
For any fixed
 point $P(\xi,\eta,\zeta)$ on the boundary $\pd'G$, denote by $P_\delta$ the intersection of the $\delta-$neighborhood of $P$ in the $(\xi,\eta,\zeta)-$space with the domain $G$. Consider the cylinder
\[H_\delta~=~[-{1\over2}-r_1,T+\frac{1}{2}+r_1]\times P_\delta.\]
We will show that there is a
small $\delta>0$ such that in the domain $H_{\delta}$, the derivatives of the solution $\W^n_\epsilon$ to the problem \eqref{pr_n} up to
order four  are bounded uniformly in $\epsilon$.

To simplify the presentation, we may
assume that in $H_\delta$ the coefficient $a_1$ depends only on $t$, and $a_i,i=2,3,4$ depend only on $\xi,\eta$ and $\zeta$,
and by introducing new coordinates $\xi',\eta'$ and $\zeta'$ in the domain $P_\delta$
if necessary,
so that the boundary:
 $$\pd'P_\delta=\pd P_\delta \cap \pd G,$$
is
a subset on the plane $\{\zeta'=0\}$, and
the inward normal direction
to $\pd'P_\delta$ coincides with that of the $\zeta'-$axis.
For simplicity, we still denote
the new coordinates by $\xi,\eta$ and $\zeta$. And
then, the boundary condition of problem \eqref{pr_n} on $[-{1\over2}-r_1,T+\frac{1}{2}+r_1]\times\pd'P_\delta$ becomes
\[{\pd \W^n_\epsilon\over\pd \zeta}\Big|_{\zeta=0}=-f^*_\epsilon.\]
For notation, we add a superscript $*$ to a function represented in the new coordinates $\xi',\eta'$ and $\zeta'$.

Note that on  the right hand side of \eqref{pr_n}, $F$ is infinitely differentiable in the region $H_\delta$, and $f$ is infinitely differentiable on the boundary $\pd H_\delta\cap\pd \Sigma$. Hence, we can choose a smooth function $X(t,\xi,\eta,\zeta)$ defined in $H_\delta$ satisfying
\[{\pd X\over\pd\zeta}\Big|_{\zeta=0}~=~f^*_\epsilon,\]
 and then,  from \eqref{pr_n} we know that the function
\[Y(t,\xi,\eta,\zeta)=\W^n_\epsilon(t,\xi,\eta,\zeta)+
X(t,\xi,\eta,\zeta),\]
satisfies the following problem in $H_\delta$:
\begin{equation}\label{eq_v}
\begin{cases}
L_1(Y)= \widetilde F^*_\epsilon,\\
\pd_\zeta Y|_{\zeta=0}~=~0,
\end{cases}
\end{equation}
where the operator
\[\begin{split}
L_1(Y)\triangleq&-a_{11}\pd_\xi^2 Y-a_{22}\pd_\eta^2 Y-a_{33}\pd_\zeta^2 Y
-2a_{12}\pd_{\xi\eta}^2 Y-2a_{13}\pd_{\xi\zeta}^2 Y-2a_{23}\pd_{\eta\zeta}^2 Y\\
&-(a_1+\epsilon)\pd_t^2 Y+\pd_t Y+b_1\pd_\xi Y+b_2\pd_\eta Y+b_3\pd_\zeta Y+[(B_\epsilon')^*+2(a_1+\epsilon)]Y,
\end{split}\]
 with the coefficients $(a_{11}, a_{22}, a_{33}, a_{12}, a_{13}, a_{23})$ being derived from $(a_2, a_3, a_4)$  through the transformation from $(\xi, \eta, \zeta)$ to $(\xi', \eta', \zeta')$, the function $\widetilde F_\epsilon^*$ has bounded derivatives  up to order four uniformly in $\epsilon$.
By using the assumption of $a_i,i=2,3,4$, there exists a positive constant $\lambda_1$, independent of $\epsilon$, such that for any $\alpha=(\alpha_1,\alpha_2,\alpha_3)\in\R^3$, we have
\[a_{11}\alpha_1^2+a_{22}\alpha_2^2+a_{33}\alpha_3^2+2a_{12}\alpha_1\alpha_2
+2a_{13}\alpha_1\alpha_3+2a_{23}\alpha_2\alpha_3\geq\lambda_1|\alpha|^2,\]
which implies that the operator $L_1$ is uniformly elliptic in $H_\delta$. Moreover, the coefficient of the zero-th order of $Y$ in $L_1(Y)$ is positive in $H_\delta$. The next main task is to study the boundedness of derivatives of the solution $Y$ to the problem \eqref{eq_v}. This is given in the following several steps by developing the idea from \cite{Ole}.\\

\underline{Step 1.}
Estimates of the first order spatial derivatives of $Y$.

Set
\begin{equation}\label{lam-1}
\Lambda_1=\rho_\delta^2(\xi,\eta,\zeta)\Big[Y_\xi^2+Y_\eta^2+Y_\zeta^2\Big]
+C_1Y^2+C_2\zeta,
\end{equation}
where $C_1$ is a positive constant to be determined later such that the  inequality
\eqref{lambda-1} given below holds, $C_2>0$ is a constant, $\rho_\delta(\xi,\eta,\zeta)$ is a smooth cut-off function, defined in $P_\delta$, satisfying
$\rho_\delta\equiv1$ in $P_{\delta/2}$, $\rho_\delta\equiv0$ in a small neighborhood of the boundary $\pd P_\delta\setminus\pd G$, and
\[{\pd \rho_\delta\over\pd \zeta}\Big|_{\zeta=0}~=~0.\]

From Lemma \ref{lem_n2-1}
we know that $\Lambda_1$ is uniformly bounded in $\epsilon$ on the boundary
 $\{t=-{1\over2}-r_1\}$ or $\{t=T+\frac{1}{2}+r_1\}$.
Next, it is easy to see that
\[{\pd \Lambda_1\over\pd \zeta}\Big|_{\zeta=0}=C_2>0,\]
which implies that $\Lambda_1$ does not
attain its maximum on the boundary $[-{1\over2}-r_1,T+\frac{1}{2}+r_1]\times\pd'P_\delta$. If the maximum of $\Lambda_1$ is attained at a point on the boundary $[-{1\over2}-r_1,T+\frac{1}{2}+r_1]\times\Big(\pd P_\delta\setminus\pd'P_\delta\Big)$, then $\rho_\delta=0$ at such point,
and
\[\Lambda_1\leq \max\{C_1V^2+C_2\zeta\}\leq C_3,\]
 by using Lemma \ref{lem_n1},
with $C_3$ being a positive constant independent of $\epsilon$.

It is easy to check that for large $C_1$, we have
\begin{equation}\label{lambda-1}
L_1(\Lambda_1)+\Lambda_1\leq C_4,\quad in~H_\delta,
\end{equation}
for a positive constant $C_4$ independent of $\epsilon$.
 Thus, if $\Lambda_1$ attains its maximum inside $H_\delta$, then
from \eqref{lambda-1}
we have
$\Lambda_1\leq C_4.$

In conclusion, we deduce
that $\Lambda_1$ is bounded uniformly in $\epsilon$ in $H_\delta$, which implies that $Y_\xi,Y_\eta$ and $Y_\zeta$ are also bounded uniformly  $\epsilon$ in $H_{\delta_1}$ for
a small constant $\delta_1<\delta$.\\

\underline{Step 2.}
Estimates of $Y_t$ and the second order spatial derivatives of $Y$.

Set
\[\Gamma(Y)~\triangleq~ Y_t-(a_1+\epsilon)\pd_t^2 Y.\]
Then, the equation given in \eqref{eq_v} can be rewritten in the following form
\[L_1(Y)=\Gamma(Y)+L_2(Y)=\widetilde F^*_\epsilon.\]
Without loss of generality, one
may assume that the coefficients of the operator $L_2$ are independent of $t$.
From \eqref{eq_v}  we know that  $\Gamma\triangleq\Gamma(Y)$ satisfies the following problem:
\begin{equation}\label{pr_gamma}
\begin{split}
&L_1(\Gamma)=\Gamma(\Gamma)+L_2(\Gamma)=\Gamma(\widetilde F^*_\epsilon),\quad in ~H_{\delta_1},\\
&\pd_\zeta \Gamma\Big|_{\zeta=0}=0.
\end{split}\end{equation}

To study the boundedness of the second order derivatives of $Y$ with respect to the variables $\xi,\eta$ and $\zeta$, similar to $\Lambda_1$,
we consider the following functional in $H_{\delta_1}$:
\begin{equation}\label{lam-2}
\Lambda_2~=~\rho_{\delta_1}[\pd_\xi^2 Y+\pd_\eta^2 Y+\pd_{\xi\eta}^2 Y+\pd_{\xi\zeta}^2 Y+\pd_{\eta\zeta}^2 Y+\Gamma^2]+C_5(Y_\xi^2+Y_\eta^2+Y_\zeta^2)+C_6\zeta,
\end{equation}
where
$C_5$ and $C_6$ are two positive constants,
and $\rho_{\delta_1}$
is a smooth cut-off
function similar to
$\rho_\delta$ given in \eqref{lam-1}.

By a computation similar to the one for $\Lambda_1$,
 we can obtain that $\Lambda_2$ is uniformly bounded in $\epsilon$ in $H_{\delta_1}$
by properly choosing $C_5$ and $C_6$,
which implies that
$\pd_\xi^2 Y,\pd_\eta^2 Y,\pd_{\xi\eta}^2 Y,$ $\pd_{\xi\zeta}^2 Y,\pd_{\eta\zeta}^2 Y$ and $\Gamma$ are uniformly bounded in $\epsilon$ in $H_{\delta_2}$ for a positive constant $\delta_2<\delta_1$.
Then, from the equation given in \eqref{eq_v} we deduce
 that $\pd_\zeta^2 Y$ is also uniformly bounded in $\epsilon$ in $H_{\delta_2}$.

Next, we consider the equation of $Y_t$:
 $$Y_t-(a_1+\epsilon)\pd_t^2 Y=\Gamma.$$
By combining with the uniform boundedness of $\Gamma$ in $H_{\delta_2}$
and of $Y_t$ at $t=-{1\over2}-r_1,~or~t=T+\frac{1}{2}+r_1$ from Lemma \ref{lem_n2-1},
it is easy to deduce that $Y_t$ is also uniformly bounded in $\epsilon$ in $H_{\delta_2}$.\\

\underline{Step 3.}
Estimates of higher order derivatives of $Y$.

Noting that $Y$ satisfies the equation:
\begin{equation}\label{est_Y}
L_2(Y)=-\Gamma+\widetilde F^*_\epsilon,\end{equation}
and $L_2$ is elliptic in $\xi,\eta,\zeta$ uniformly in $\epsilon$. In order to use the Schauder estimates of elliptic equations to study the third and fourth order derivatives of $Y$, one needs to estimate the derivatives of $\Gamma$ in $\xi,\eta$ and $\zeta$ up to order three. Similarly, from
\begin{equation}\label{est_Y1}
L_2(\Gamma)=-\Gamma(\Gamma)+\Gamma(\widetilde F^*_\epsilon),\end{equation}
we need to estimate the derivatives of $\Gamma(\Gamma)$ in $\xi,\eta$ and $\zeta$ up to  order two.

Since  $\Gamma$ is uniformly bounded in $\epsilon$ in $H_{\delta_2}$ and satisfies the problem \eqref{pr_gamma}, as in \cite{OA1}, by studying
some functionals of $\Gamma$ similar to $\Lambda_1$ and $\Lambda_2$ of $Y$ in the region $H_{\delta_2}$,  we can obtain the boundedness of $\mathcal{F}(\Gamma)$ in $H_{\delta_3}$ uniformly in $\epsilon$ for a positive constant $\delta_3<\delta_2$, with
\begin{equation}\label{def_F}
\mathcal{F}(Y)=\Big(Y_\xi,~Y_\eta,~Y_\zeta,~Y_t,~\pd_\xi^2Y,~\pd_\eta^2Y,
~\pd_{\xi\eta}^2Y,~\pd_{\xi\zeta}^2Y,~\pd_{\eta\zeta}^2Y,
~\pd_\zeta^2Y,~\Gamma(Y)\Big).
\end{equation}

Similarly, for suitable $a_1$, similar arguments
holds  for  $\Gamma_t$ and $\Gamma_{tt}$ so that we
can obtain the uniform boundeness of $\mathcal{F}(\Gamma_t)$ and $\mathcal{F}(\Gamma_{tt})$ in $H_{\delta_4}$ for some positive constant $\delta_4<\delta_3$.

From these uniform estimates,
we deduce that in $H_{\delta_4}$, both of the third and fourth order derivatives of $Y$ containing more than one order
 differentiation in $t$ and the derivatives of $\Gamma(\Gamma)$ with respect to $\xi,\eta$ and $\zeta$ up to
order two are
bounded  uniformly in $\epsilon$. Therefore, from \eqref{est_Y} and \eqref{est_Y1} we know that the derivatives of  $Y$ up to order four  are uniformly
bounded  in $\epsilon$ in $H_{\delta_4}$.
This completes  the proof of the lemma.
\end{proof}

\vspace{.15in}

{\bf Acknowledgements :}
The first two authors' research was supported in part by
National Natural Science Foundation of China (NNSFC) under Grants
No. 10971134, No. 11031001 and No. 91230102. The last author's research was supported by the General Research Fund of Hong Kong,
CityU No. 103713.
We would like to thank Dr. Fang Yu for her valuable discussion on this problem.


\end{document}